\documentclass{elsarticle}

\usepackage{amsmath,amssymb,amsfonts,amscd}
\usepackage[all]{xy}
\usepackage{xcolor}

\newtheorem{theorem}{Theorem}[section]
\newtheorem{definition}[theorem]{Definition}
\newtheorem{lemma}[theorem]{Lemma}
\newtheorem{corollary}[theorem]{Corollary}

\newtheorem{example}[theorem]{Example}
\newtheorem{remark}[theorem]{Remark}
\newproof{proof}{\it Proof}

\begin{document}

\begin{frontmatter}



\title{A space of generalized 
Brownian motion path-valued  continuous functions
with application}

 \author[label3]{Jae Gil Choi}
\ead{jgchoi@dankook.ac.kr}
\author[label3]{Seung Jun Chang}
\ead{sejchang@dankook.ac.kr}

\address[label3]{Department of Mathematics,  
                 Dankook University,
                 Cheonan 330-714, 
                 Korea}
\cortext[cor1]{Corresponding author}

 
\begin{abstract}
In this paper, we introduce the paths space $\mathcal C_0^{\mathrm{gBm}}$ 
which is consists of  generalized Brownian motion path-valued 
continuous functions on $[0,T]$. We next  present several relevant examples 
of the paths space integral. We then discuss the concept of the
analytic Feynman integration theory for functionals $F$ on the   paths space   $\mathcal C_0^{\mathrm{gBm}}$.
\end{abstract}

\begin{keyword}
Generalized Brownian motion process  \sep 
Paley--Wiener--Zygmund stochastic integral  \sep 
paths space  \sep
analytic paths space Feynman integral.

\vspace{.3cm}
\MSC[2010]   60J65 \sep 28C20 \sep  46G12 

\end{keyword}

\end{frontmatter}


\setcounter{equation}{0}
\section{Introduction}

\par
Let $(\mathbb B,\gamma)$ denote an abstract Wiener space and let 
$\mathcal C_0(\mathbb B)$ be the space of $\mathbb B$-valued continuous 
functions $\mathfrak x$  which are defined   on  $[0,T]$ 
with $\mathfrak x(0) = 0$,
see \cite{KLe73}. In \cite{Ryu92}, Ryu  improved several theories 
on $\mathcal C_0(\mathbb B)$  which are developed in classical and abstract Wiener spaces. 
Since then the concepts of the analytic Feynman integral and the analytic Fourier--Feynman 
transform, and related topics have been  developed on the  Wiener paths space 
$\mathcal C_0(\mathbb B)$,  extensively; 
references include \cite{CCY02,CKSY02,CKSY10,Cho08,Kim05,KK06}.
In \cite{Ryu92}, Ryu  suggested a  cylinder   measure  $m_{\mathbb B}$ on 
the space $\mathcal C_0(\mathbb B)$  and  constructed  the general Wiener integration 
theorem: given a multi-dimensional tuple $(t_1,t_2,\ldots,t_n)\in \mathbb R^n$ with
$0=t_0<t_1<t_2< \cdots <t_n\le T$,  and  a Borel measurable function  $f:\mathbb R^n \to \mathbb C$,
\begin{equation}\label{eq:Ryu}
\begin{aligned}
&\int_{\mathcal C_0(\mathbb B)} 
f(\mathfrak x(t_1),\mathfrak x(t_2),\ldots, \mathfrak x(t_n)) 
dm_{\mathbb B}( \mathfrak x)\\ 
&=\int_{\mathbb B^n} f\bigg(\sqrt{t_1-t_0} x_1 , 
\sqrt{t_1-t_0} x_1 +\sqrt{t_2-t_1} x_2 ,
\ldots,\sum_{j=1}^n \sqrt{t_{j}-t_{j-1}}x_j\bigg)\\
&\qquad\times
d\gamma^n(x_1,\ldots,x_n)  
\end{aligned}
\end{equation}
in the sense that if either side exists,  both sides exist and 
the equality holds. 
The concrete formulation of the  cylinder  measure  $m_{\mathbb B}$
and the applications to the theory of analytic Feynman integral, 
see \cite{CCY02,CKSY02,CKSY10,Cho08,Kim05,KK06,Ryu92}  and the references cited 
therein. 
Equation \eqref{eq:Ryu} will be evaluated in Section \ref{sec:MMo} with heuristic observations.

\par
On the other hand, in \cite{CCK15-I,CCK15-II,CCK16,CCS03,CCS10,CS03,CC12}, 
the authors defined the generalized analytic Feynman integral and the 
generalized analytic Fourier--Feynman transform  on the function space 
$C_{a,b}[0,T]$, and studied their properties with related topics. The function 
space $C_{a,b}[0,T]$, induced by a generalized Brownian motion process (GBMP),
 was introduced by Yeh in \cite{Yeh71}, 
and was used extensively in \cite{CCS07,CC96,CChungS09,CLC15,CCC13}.

\par
A GBMP on a probability space $(\Omega,\Sigma, P)$ and a time interval $[0,T]$ 
is a Gaussian  process  $Y \equiv\{Y_t\}_{t\in [0,T]}$ such that $Y_0=c$ almost
surely for some constant $c \in \mathbb  R$, and for any set of time  
moments $0= t_0 < t_1< \cdots<t_n \le T$ and any Borel set $B\subset \mathbb R^n$, 
the measure $P(I_{t_1,\ldots,t_n,B})$ of the cylinder set $I_{t_1,\ldots,t_n,B}$ of 
the form 
$I_{t_1,\ldots,t_n,B} 
=\big\{\omega\in \Omega:(Y_{t_1}(\omega),\ldots,Y_{t_n}(\omega))\in B \big\}$
is given by
\[
P(I_{t_1,\ldots,t_n,B})=\int_{B}K_n(\vec t,\vec \eta) 
d\eta_1\cdots d\eta_n 
\]
where 
\[
\begin{aligned}
K_n(\vec t,\vec \eta) 
&=\bigg( (2\pi )^n \prod\limits_{j=1}^n\big(b(t_j)-b(t_{j-1})\big) \bigg)^{-1/2} \\ 
&\quad\times 
\exp \bigg\{- \frac 12 \sum\limits_{j=1}^n
   \frac {((\eta_j-a(t_j))-(\eta_{j-1}-a(t_{j-1})))^2}
         {b(t_j)-b(t_{j-1})} \bigg\},
\end{aligned}
\]
and where $\eta_0=c$,    $a(t)$ is a continuous real-valued function on  $[0,T]$, 
and $b(t)$ is a  increasing continuous real-valued function on $[0, T]$. 
Thus, the  GBMP $Y$ is determined by the continuous functions $a(\cdot)$ 
and $b(\cdot)$. For more details, see \cite{Yeh71,Yeh73}. Note that when $c=0$, 
$a(t)\equiv 0$ and $b(t)=t$ on $[0,T]$, the GBMP reduces a standard 
Brownian motion (Wiener process).

\par
We set $c = a(0)=b(0)=0$. Then the function space $C_{a,b}[0,T]$ 
induced by the GBMP $Y$ determined by the  $a(\cdot)$ and $b(\cdot)$  
can be considered as the space of continuous sample paths of $Y$, see 
\cite{CCK15-I,CCK15-II,CCK16,CCS03,CCS07,CCS10,CC96,CChungS09,CLC15,CS03,CC12,CCC13},
and one can see that for each $t\in [0,T]$,
\[
e_t(x) \sim N\big(a(t), \,b(t)\big),
\]
where $e_t:C_{a,b}[0,T]\times[0,T]\to\mathbb R$ is the coordinate evaluation map
given by $e_t(x)=x(t)$ and $N(m, \sigma^2)$ denotes the normal distribution with 
mean $m$ and variance $\sigma^2$. We are obliged  to point out that a standard 
Brownian motion is stationary in time and is free of drift, whereas a GBMP is 
generally not stationary in time  and is subject to a drift $a(t)$.

\par
In this paper, we thus first attempt to construct 
the paths space $\mathcal C_{a,b}^{\mathrm{gBm}} \equiv \mathcal C_{a,b}^{\mathrm{gBm}}(C_{a,b}[0,T])$  
which is consists of  generalized 
Brownian motion path-valued continuous functions on $[0,T]$. 
We next  present several relevant 
examples of the paths space integral. As an application, 
we then discuss the concept of the analytic Feynman integration theory 
for functionals $F$ on the  paths space   $\mathcal C_0^{\mathrm{gBm}}$.
To do this we establish  the existence of the  
analytic paths space Feynman integral of bounded cylinder functionals $F$ of the form 
\[
F(\mathfrak x) 
=\int_{\mathbb R^{mn}} \exp\bigg\{ i \sum\limits_{j=1}^m
\sum\limits_{k=1}^n (g_j,\mathfrak x(s_k))^{\sim} v_{j,k} \bigg\}
d \nu (\vec v),\qquad \frak x \in \mathcal C_0^{\mathrm{gBm}} 
\]
where $\nu$ is a complex Borel measure on $\mathbb R^{mn}$
and $(g,\mathfrak x(s))^{\sim}$ denotes the Paley--Wiener--Zygmund (henceforth, PWZ) 
stochastic integral.
In Section \ref{sec:MMo} below, we present a more detailed survey of paths space
 and a motivation of the topic in this paper.

\setcounter{equation}{0}
\section{Preliminaries}\label{sec:path}

In this section, we  present  the brief backgrounds  
which are needed in the following sections.

\par
Let  $a(t)$ be an absolutely continuous real-valued  function on $[0,T]$ with $a(0)=0$ 
and  $a'(t)\in L^2[0,T]$,  and let $b(t)$ be a strictly increasing, continuously 
differentiable real-valued function with $b(0)=0$ and $b'(t) >0$ for each $t\in[0,T]$. 
The  GBMP   $Y$ determined by $a(t)$ and $b(t)$ is a Gaussian process  with mean function  
$a(t)$ and covariance function $r(s,t)=\min\{b(s),b(t)\}$.  For more details, 
see \cite{CCS03,CChungS09,CS03,Yeh71,Yeh73}. Applying  \cite[Theorem 14.2]{Yeh73},
one can construct a probability measure space $(C_{a,b}[0,T],\mathcal{B}(C_{a,b}[0,T]),\mu)$
where  $C_{a,b}[0,T]$ is the space of continuous sample paths of (a separable version of) 
the GBMP $Y$ (it is equivalent to the Banach space of continuous functions $x$ on $[0,T]$ 
with $x(0)=0$ under the sup norm) and  $\mathcal B(C_{a,b}[0,T])$ is the Borel $\sigma$-field 
of $C_{a,b}[0,T]$ induced by the sup norm. We then complete this function space to obtain  
the complete probability  measure space $(C_{a,b}[0,T],\mathcal W(C_{a,b}[0,T]),\mu)$
where $\mathcal W(C_{a,b}[0,T])$ is the set of all Wiener  
measurable subsets of $C_{a,b}[0,T]$.

\begin{remark} 
The coordinate process $e: C_{a,b}[0,T]\times[0,T]\to\mathbb R$ defined by  
$e(x,t)\equiv e_t(x)=x(t)$  is also the GBMP determined by $a(t)$ and $b(t)$.  
\end{remark}

\begin{remark} 
Let  $C_{a,b}^n[0,T]$ be the product of $n$ copies  of  $C_{a,b}[0,T]$. Since  the space  
$C_{a,b}[0,T]$ endowed with the uniform topology  is  separable, the Borel $\sigma$-field 
$\mathcal B(C_{a,b}^n[0,T])$ on $C_{a,b}^n[0,T]$ coincides with the product $\sigma$-field 
$\otimes^n \mathcal B(C_{a,b}[0,T])$, the product of $n$ copies of $\mathcal B(C_{a,b}[0,T])$.
From this fact we see that
\[
\mathcal W (C_{a,b}^n [0,T])= \overline{\mathcal B (C_{a,b}^n [0,T])}
=\overline{\otimes^n\mathcal W (C_{a,b} [0,T])}
\]
where $\mathcal W (C_{a,b}^n [0,T])$ denotes the $\sigma$-field consisting of 
all Wiener measurable subsets of the product function space $C_{a,b}^n[0,T]$,
$\otimes^n\mathcal W (C_{a,b} [0,T])$ is the product of $n$ copies  of  the 
$\sigma$-field $\mathcal W(C_{a,b}[0,T])$ on $C_{a,b}[0,T]$, and $\overline{\mathcal S}$ 
denotes the complete $\sigma$-field generated by a $\sigma$-field $\mathcal S$.
\end{remark}

\par
Let $L_{a,b}^2[0,T]$ be the space of functions on $[0,T]$ which are Lebesgue 
measurable and square integrable with respect to the Lebesgue--Stieltjes measures 
on $[0,T]$ induced by  $a(\cdot)$ and $b(\cdot)$; i.e.,
\[
L_{a,b}^2[0,T]
=\bigg\{   v :  \int_{0}^{T} v^2 (s) db(s)  <+\infty  \hbox{ and }
                \int_0^T v^2 (s) d |a|(s) < +\infty \bigg\}
\]
where $|a|(\cdot)$ denotes the   total variation   function of $a(\cdot)$.
Then $L_{a,b}^2[0,T]$ is a separable Hilbert space with inner product defined by
\[
(u,v)_{a,b}=\int_0^T u(t)v(t)dm_{|a|,b}(t)\equiv \int_0^T u(t)v(t)d[b(t)+|a|(t)],
\]
where $m_{|a|,b}$ denotes  the Lebesgue--Stieltjes measure induced by $|a|(\cdot)$
and $b(\cdot)$. In particular, note that $\| u\|_{a,b}\equiv\sqrt{(u,u)_{a,b}} =0$
if and only if $u(t)=0$ a.e. on $[0,T]$.
Furthermore, $(L_{a,b}^2[0,T],\|\cdot\|_{a,b})$  is a separable Hilbert space.

\par
Next, let
\[
C_{a,b}'[0,T]
 =\bigg\{ w \in C_{a,b}[0,T] : w(t)=\int_0^t z(s) d b(s)
\hbox{  for some   } z \in L_{a,b}^2[0,T]  \bigg\}.
\]
For $w\in C_{a,b}'[0,T]$, with $w(t)=\int_0^t z(s) d b(s)$ for $t\in [0,T]$,
let $D: C_{a,b}'[0,T] \to L_{a,b}^2[0,T]$ be defined by the formula
\begin{equation}\label{eq:Dt}
Dw(t)= z(t)=\frac{w'(t)}{b'(t)}.
\end{equation}
Then $C_{a,b}' \equiv C_{a,b}'[0,T]$ with inner product
\[
(w_1, w_2)_{C_{a,b}'}
=\int_0^T  Dw_1(t)  Dw_2(t)  d b(t)
\]
is also a separable  Hilbert space.

\begin{remark}
Note that  the two separable Hilbert spaces $L_{a,b}^2[0,T]$ and $C_{a,b}'[0,T]$
are (topologically) homeomorphic under the linear operator $D$ given by equation 
\eqref{eq:Dt}. The inverse operator of $D$ is given by 
$(D^{-1}z)(t)=\int_0^t z(s) d b(s)$ for $t\in [0,T]$.
But the linear operator $D$ is not isometric.
\end{remark}

\par
In this paper, in addition to the conditions put on $a(t)$ above,
we now add the condition
\begin{equation}\label{eq:new-cc2}
\int_0^T |a'(t)|^2 d|a|(t)< +\infty.
\end{equation}
Then, the function $a: [0,T]\to\mathbb R$ satisfies the condition \eqref{eq:new-cc2} 
if and only if $a(\cdot)$ is an element of $C_{a,b}'[0,T]$. Under the  condition 
\eqref{eq:new-cc2}, we observe that for each  $w\in C_{a,b}'[0,T]$ with $Dw=z$,
\[
(w,a)_{C_{a,b}'}=\int_0^T Dw(t) Da(t) db(t)
=\int_0^T z(t)\frac{a'(t)}{b'(t)}db(t)=\int_0^T z(t)da(t).
\]

\par
For each $w\in C_{a,b}'[0,T]$  and $x\in C_{a,b}[0,T]$,  we let $(w,x)^{\sim}$ denote 
the  PWZ stochastic integral \cite{CCK15-I,CCC13}. It is known that  for each 
$w\in C_{a,b}'[0,T]$, the PWZ stochastic integral  $(w,x)^{\sim}$  exists for s-a.e. 
$x\in C_{a.b}[0,T]$ and it is a Gaussian random variable with mean $(w,a)_{C_{a,b}'}$  
and variance $\|w\|_{C_{a,b}'}^2$. It also  follows  that for $w, x\in C_{a,b}'[0,T]$, 
\begin{equation}\label{eq:pwz-meam-type}
(w,x)^{\sim}=(w,x)_{C_{a,b}'} 
\end{equation}
and that  for  $w_1,w_2\in C_{a,b}'[0,T]$,
\begin{equation}\label{eq:pwa-cov-semi}
\int_{C_{a,b}[0,T]}(w_1,x)^{\sim}(w_2,x)^{\sim}d\mu(x)
=(w_1,w_2)_{C_{a,b}'}+(w_1,a)_{C_{a,b}'}(w_2,a)_{C_{a,b}'}.
\end{equation}
Thus the random variable $(w,x)^{\sim}$ is normally distributed with
\[
(w,x)^{\sim} \sim N\Big((w,a)_{C_{a,b}'},\|w\|_{C_{a,b}'}^2\Big).
\]
Furthermore, if $Dw=z\in L_{a,b}^2[0,T]$ is of bounded variation on $[0,T]$,  the PWZ 
stochastic integral $(w,x)^{\sim}$ equals the Riemann--Stieltjes integral $\int_0^T z(t)dx(t)$.

\par
For each $t\in[0,T]$, let
\begin{equation}\label{rkernel}
\beta_t(s)=\int_0^s \chi_{[0,t]}(\tau) d b(\tau)
=\begin{cases}
    b(s), \quad & 0\le s \le t\\
    b(t), \quad & t\le s \le T
\end{cases}.
\end{equation}
Then the family of functions $\{\beta_t: 0\le t\le T\}$
from $C_{a,b}'[0,T]$ has the reproducing property
\[
(w,\beta_t)_{C_{a,b}'}=w(t)
\]
for all $w\in C_{a,b}'[0,T]$.
Note that for any $s,t\in[1,2]$, $\beta_t(s)=\min\{b(s), b(t)\}$,
the covariance function associated with the 
generalized Brownian motion $Y$ used in this paper.
We also note that for each $x\in C_{a,b}[0,T]$,
\begin{equation}\label{kernel-repre}
x(t)=\int_0^T \chi_{[0,t]}(\tau)dx(\tau)=(\beta_t,x)^{\sim}.
\end{equation}

\par
Using the change of variable theorem, it follows the function space integration 
formula:
\begin{equation}\label{eq:int-formula-cab}
\begin{aligned}
&\int_{C_{a,b}[0,T]}\exp\{\rho (w,x)^{\sim}\}d\mu(x) 
=\exp\bigg\{\frac{\rho^2}{2}\|w\|_{C_{a,b}'}^2 
+\rho(w ,a)_{C_{a,b}'}\bigg\}
\end{aligned}
\end{equation}
for every  $\rho>0$.

\setcounter{equation}{0}
\section{Motivations}\label{sec:MMo}
\subsection{Survey on the classical Wiener space $C_0[0,T]$}\label{sec:motivation}
  
Given a positive real $T>0$, let  $C_0[0,T]$ denote one-parameter Wiener space, 
that is, the space of  all  real-valued continuous functions $x$ on the  
interval $[0,T]$   with $x(0)=0$. Let $\mathcal{M}$ denote the class of  all 
Wiener measurable subsets of $C_0[0,T]$ and  let $m_w$ denote Wiener measure.
Then, as is well-known, $(C_0[0,T],\mathcal{M},m_w)$ is a complete probability measure space.
The coordinate process  $\widetilde W \equiv \{\widetilde W_t\}_{t\in [0,T]}$
given  by $\widetilde W_t(x) = x(t)$ on $C_0[0,T]\times[0,T]$
is a standard Brownian motion (henceforth, SBM). Thus Wiener measure  $m_w$ is a Gaussian measure on $C_0[0,T]$ 
with mean zero and covariance function $r(s,t)=\min\{s,t\}$  in view of following
illustration.

\par
The SBM (equivalently, Wiener  process) on a probability 
space $(\Omega,\Sigma, \mathbf{p})$ and a time interval $[0,T]$ is a Gaussian  process  
$W \equiv\{W_t\}_{t\in [0,T]}$ such that $W_0=0$ almost surely, and for any set 
of time  moments $0= t_0 < t_1< \cdots<t_n \le T$ and any Borel set $B\subset \mathbb R^n$, 
the measure $\mathbf{p}(I_{t_1,\ldots,t_n,B})$ of the cylinder set $I_{t_1,\ldots,t_n,B}$ of 
the form
\[ 
I_{t_1,\ldots,t_n,B} 
=\big\{\omega\in \Omega:(W_{t_1}(\omega),\ldots,W_{t_n}(\omega))\in B \big\}
\]
is given by
\begin{equation}\label{eq:basic001}
\bigg( (2\pi )^n \prod\limits_{j=1}^n(t_j-t_{j-1}) \bigg)^{-1/2}   
\int_{B}\exp \bigg\{-\frac12\sum\limits_{j=1}^n\frac{[u_j-u_{j-1}]^2}{t_j-t_{j-1}} \bigg\}
du_1\cdots du_n
\end{equation}
where $u_0=0$.
The coordinate process $\widetilde W: C_{0}[0,T]\times[0,T]\to\mathbb R$ defined by  
$\widetilde W(x,t)\equiv \widetilde W_t(x)=x(t)$  is also a   SBM. 
Thus the Wiener space $C_0[0,T]$ can be considered as the space of all sample paths 
of a Brownian motion. We observe  that for any $t_1,t_2\in[0,T]$ with 
$t_1<t_2$,
\[
\widetilde W(x,t_2)-\widetilde W(x,t_1)\sim N(0, t_2-t_1),
\]
where $N(m, \sigma^2)$ denotes the normal distribution with 
mean $m$ and variance $\sigma^2$.
 Given the time  moments $0= t_0 < t_1< \cdots<t_n \le T$,
define a function 
$P_{(t_1,\ldots,t_n)}: C_{0}[0,T]\to \mathbb R^n$ by 
$P_{(t_1,\ldots,t_n)}(x)=(x(t_1),\ldots,x(t_n))$.
Then the Wiener measure $m_w(I_{t_1,\ldots,t_n,B})$ of the cylinder set 
$I_{t_1,\ldots,t_n,B} =\{x\in C_0[0,T]:P_{(t_1,\ldots,t_n)}(x)\in B \}$
with a Borel set $B$ in $\mathbb R^n$ 
is given by  \eqref{eq:basic001}. Furthermore,   the probability distribution 
$m_w\circ P_{(t_1,\ldots,t_n)}^{-1}$ and the Lebesgue measure $m_L^n$ on $\mathbb R^n$ 
are mutually absolutely continuous. Thus the Radon--Nikodym derivative
of $m_w\circ P_{(t_1,\ldots,t_n)}^{-1}$ with respect to $m_L^n$ is given by
\begin{equation}\label{eq:basic002}
\begin{aligned}
&\frac{m_w\circ P_{(t_1,\ldots,t_n)}^{-1}}{m_L^n}(u_1,\ldots,u_n)\\
&=\bigg( (2\pi )^n \prod\limits_{j=1}^n(t_j-t_{j-1}) \bigg)^{-1/2}   
 \exp \bigg\{-\frac12\sum\limits_{j=1}^n\frac{[u_j-u_{j-1}]^2}{t_j-t_{j-1}} \bigg\} 
\end{aligned}
\end{equation}
with $u_0=0$. In fact, for any subset $E$ of $\mathbb R^n$,
$E$ is Lebesgue measurable if and only if $P_{(t_1,\ldots,t_n)}^{-1}(E)$
is Wiener measurable. 
For more details, see \cite{CR88,JS79} and  references cited therein.

\par
For each $v\in L_2[0,T]$  and $x\in C_{0}[0,T]$,  we let $\langle{v,x}\rangle$ denote 
the Paley--Wiener--Zygmund (PWZ) stochastic integral \cite{JS81,PWZ33,PS88}. It is known 
that  for each $v\in L_2[0,T]$, the PWZ stochastic integral  $\langle{v,x}\rangle$  exists 
for $m_w$-a.s. $x\in C_{0}[0,T]$ and it is a Gaussian random variable with mean $0$  and 
variance $\|v\|_{2}^2$. It is also  known that for  $v_1,v_2\in L_2[0,T]$, 
\begin{equation}\label{eq:pwa-cov-semi-L2}
\int_{C_{0}[0,T]}\langle{v_1,x}\rangle\langle{v_2,x}\rangle d m_w(x)
=(v_1,v_2)_{2}
\end{equation}
where $(\cdot,\cdot)_2$ denotes the $L_2$-inner product. Furthermore, if $v\in L_{2}[0,T]$ 
is of bounded variation on $[0,T]$, then the PWZ stochastic integral $\langle{v,x}\rangle$ 
equals the Riemann--Stieltjes integral $\int_0^T v(t)dx(t)$.

\par
Let 
\[
C_{0}'[0,T]=\bigg\{ w \in C_{0}[0,T] : w(t)=\int_0^tv(\tau)d\tau
\hbox{  for some   } v \in L_2[0,T]  \bigg\}.
\]
Then the  Cameron--Martin space  $C_{0}' \equiv C_{0}'[0,T]$ is a real separable infinite dimensional 
Hilbert space with inner product
\[
(w_1, w_2)_{C_{0}'}
=\int_0^T Dw_1(\tau)Dw_2(\tau)d\tau 
\]
where $Dw(\tau)= \frac{dw}{d\tau}(\tau)$.
 Given any $w\in C_0'[0,T]$, we use the notation
$(w,x)^{\sim}$ to denote the   PWZ stochastic integral $\langle{Dw,x}\rangle$. 
Then  for $w, x\in C_{0}'[0,T]$, 
$(w,x)^{\sim}=(w,x)_{C_{0}'} $
and  equation \eqref{eq:pwa-cov-semi-L2} above can be rewritten as follows: 
for  $w_1,w_2\in C_{0}'[0,T]$,
\begin{equation}\label{eq:pwa-cov-semi}
\int_{C_{0}[0,T]}(w_1,x)^{\sim}(w_2,x)^{\sim}d m_w(x) 
=(w_1,w_2)_{C_{0}'}.
\end{equation}

\par
For each $t\in[0,T]$, let
\begin{equation}\label{rkernel}
\beta_t(s)=\int_0^s \chi_{[0,t]}(\tau) d\tau
=\begin{cases}
    s, \quad & 0\le s \le t\\
    t, \quad & t< s \le T
\end{cases}.
\end{equation}
Then the family of functions $\{\beta_t: 0\le t\le T\}$
from $C_{0}'[0,T]$ has the reproducing property
\[
(w,\beta_t)_{C_{0}'}=w(t)
\]
for all $w\in C_{0}'[0,T]$. Note that $\beta_t(s)=\min\{s,t\}$,
the covariance function of the  Brownian motion $\widetilde W$ discussed above.
We also note that for each $(x,t)\in C_{0}[0,T]\times [0,T]$,
\begin{equation}\label{rkernel-path}
\widetilde W(x,t)=x(t)=\int_0^T \chi_{[0,t]}(\tau)dx(\tau)=(\beta_t,x)^{\sim}.
\end{equation}

\par
We will discuss the Wiener integral of three kinds of tame functions on $C_{0}[0,T]$.
Given an $n$-tuple $(t_1,\ldots,t_n)$ of  time  moments with $0= t_0 < t_1< \cdots<t_n \le T$, 
let $F: C_{0}[0,T]\to \mathbb C$ be a tame function  given by
\[
F(x)=f(x(t_1), x(t_2),\ldots, x(t_n))
\]
where $f:\mathbb R^n\to \mathbb C$ is a Lebesgue measurable function.
Then applying equation  \eqref{eq:basic002},  it follows that
\begin{equation}\label{wint-1st}
\begin{aligned}
&\int_{C_0[0,T]}F (x) d m_w(x) \\
& =\int_{C_0[0,T]}f(x(t_1),x(t_2), \ldots, x(t_n))d m_w(x)\\
&=\bigg( (2\pi )^n \prod\limits_{j=1}^n(t_j-t_{j-1}) \bigg)^{-1/2}\\
&\quad \times   
\int_{\mathbb R^n}f(u_1,u_2,\ldots,u_n)
\exp \bigg\{- \frac 12 \sum\limits_{j=1}^n
\frac{[u_j-u_{j-1}]^2}{t_j-t_{j-1}} \bigg\}
d\vec u
\end{aligned} 
\end{equation}
where $u_0=0$. 
For each $j\in\{1,2,\ldots,n\}$, let $\alpha_j(\tau)=\chi_{[0,t_j]}(\tau)$.
Then $\{\alpha_1,\alpha_2,\ldots, \alpha_n\}$ is a linearly independent set of functions 
in $L_2[0,T]$, and $\langle{\alpha_j,x}\rangle=x(t_j)$ for each $j\in\{1,2,\ldots,n\}$.

\par
Next we consider the second kind of  tame function  $F$ on $C_0[0,T]$
given by 
\[
F(x)=f(x(t_1), x(t_2)-x(t_1), \ldots, x(t_n)-x(t_{n-1})).
\]
For each $j\in \{1,2,\ldots,n\}$, let 
\[
X_j(x)=x(t_j)-x(t_{j-1}).
\]
Then $X_j$'s form a set of  independent Gaussian random variables 
such that $X_j \sim N (0, t_j-t_{j-1})$ for each  $j\in \{1,2,\ldots,n\}$. 
 Thus,  by the change of variables theorem, it follows that
\begin{equation}\label{wint-2nd}
\begin{aligned}
&\int_{C_0[0,T]} F (x) d m_w(x)\\
&=\int_{C_0[0,T]}f(x(t_1),x(t_2)-x(t_1),  \ldots, x(t_n)-x(t_{n-1}))d m_w(x)\\
&=\bigg( (2\pi )^n \prod\limits_{j=1}^n(t_j-t_{j-1}) \bigg)^{-1/2} 
 \int_{\mathbb R^n}f(v_1,\ldots, v_n )
\exp \bigg\{- \frac 12 \sum\limits_{j=1}^n
   \frac {v_j^2}{t_j-t_{j-1}} \bigg\}
d\vec v.
\end{aligned} 
\end{equation}
For each $j\in\{1,2,\ldots,n\}$, let $\delta_j(\tau)=\chi_{[t_{j-1},t_j]}(\tau)$.
Then $\{\delta_1,\delta_2,\ldots, \delta_n\}$ is an orthogonal  set of functions 
in $L_2[0,T]$, and $\langle{\delta_j,x}\rangle=X_j(x)$ for each $j\in\{1,2,\ldots,n\}$.

\par
Finally the third kind of tame function $F$ we consider  is given by 
\[
F(x)=f\bigg(\frac{x(t_1)}{\sqrt{t_1}}, \frac{x(t_2)-x(t_1)}{\sqrt{t_2-t_{1}}}, 
\ldots, \frac{x(t_n)-x(t_{n-1})}{\sqrt{t_n-t_{n-1}}}\bigg).
\]
For each $j\in \{1,2,\ldots,n\}$, let 
\[
Y_j(x)=\frac{x(t_j)-x(t_{j-1})}{\sqrt{t_j-t_{j-1}}}.
\]
Then $Y_j$'s form a set of  i.i.d. Gaussian random variables. We note that 
for each  $j\in \{1,2,\ldots,n\}$, $Y_j \sim N (0,1)$.
Thus,  by the change of variables theorem, it follows that
\[
\begin{aligned}
&\int_{C_0[0,T]} F (x) d m_w(x)\\
&=\int_{C_0[0,T]}f\bigg(\frac{x(t_1)}{\sqrt{t_1}}, \frac{x(t_2)-x(t_1)}{\sqrt{t_2-t_{1 }}}, 
\ldots, \frac{x(t_n)-x(t_{n-1})}{\sqrt{t_n-t_{n-1}}}\bigg)d m_w(x)\\
&=\int_{\mathbb R^n}f(w_1,\ldots,w_n )d \gamma_{G}(\vec w),
\end{aligned} 
\]
where $\gamma_G$ is the standard Gaussian measure on $\mathbb R^n$ given by
\begin{equation}\label{sgm}
d \gamma_{G}(\vec w)
=(2\pi )^{-n/2}\exp \bigg\{- \frac 12 \sum\limits_{j=1}^n{w_j^2} \bigg\}dw_1\cdots dw_n.
\end{equation}
For each $j\in\{1,2,\ldots,n\}$, 
let $\varepsilon_j(\tau)=(t_j-t_{j-1})^{-1/2}\chi_{[t_{j-1},t_j]}(\tau)$.
Then $\{\varepsilon_1,\varepsilon_2,\ldots,\varepsilon_n\}$ is an orthonormal  set of functions 
in $L_2[0,T]$, and $\langle{\varepsilon_j,x}\rangle=Y_j(x)$ for each $j\in\{1,2,\ldots,n\}$.

\par
In the last expression of \eqref{wint-1st}, we consider the 
following transformation $S:\mathbb R^n \to\mathbb R^n$ given by 
\begin{equation}\label{Tnew-pre}
(v_1,v_2,\ldots,v_n)
\,\,\,\stackrel{S}{\longmapsto}\,\,\,
 (u_1,u_2,\ldots,u_n)=(v_1,v_1+v_2, \ldots,v_1+v_2 +\cdots+v_n).
\end{equation}
Then it follows that
\[
v_j =  u_j-u_{j-1}
\]
for all $j\in\{1,\ldots,n\}$ and
\[
\mathcal J \bigg(\frac{u_1,\ldots, u_n}{v_1,\ldots, v_n}\bigg)=1  
\]
where $\mathcal J$ denotes the Jacobi symbol.
In these setting, equation \eqref{wint-1st} can be rewritten by
\begin{equation}\label{eq:Ryu-other}
\begin{aligned}
&\int_{C_0[0,T]}f(x(t_1),x(t_2), \ldots, x(t_n))d m_w(x)\\
&=\bigg(\prod_{j=1}^n2\pi(t_j-t_{j-1})\bigg)^{-1/2}\\
&\quad \times
\int_{\mathbb R^n}  
f(v_1,v_1+v_2, \ldots, v_1+v_2+ \cdots +v_n)
\exp\bigg\{-\frac{1}{2}\sum_{j=1}^n \frac{v_j^2}{t_j-t_{j-1}}\bigg\}d\vec v.
\end{aligned}
\end{equation}

\par
Next,  in the last expression of \eqref{wint-1st}, 
we consider the following transformation
$T:\mathbb R^n \to\mathbb R^n$ given by 
\begin{equation}\label{Told-pre}
\begin{aligned}
&(w_1,w_2,\ldots,w_n) 
\,\,\,\stackrel{T}{\longmapsto}\,\,\,\\
&(u_1,u_2,\ldots,u_n) \\
& = \bigg( \sqrt{t_{1}-t_{0}} w_1, 
\sqrt{t_{1}-t_{0}} w_1+\sqrt{t_{2}-t_{1}} w_2,
\ldots, \sum_{j=1}^n\sqrt{t_{j}-t_{j-1}} w_j\bigg). 
\end{aligned}
\end{equation}
Then it follows that
\[
w_j = \frac{u_j-u_{j-1}}{\sqrt{t_j-t_{j-1}}}
\]
for each $j\in\{1,\ldots,n\}$ and
\[
\mathcal J\bigg(\frac{u_1,\ldots, u_n}{w_1,\ldots, w_n}\bigg) 
 =\prod_{j=1}^n \sqrt{t_j-t_{j-1}}.
\]
In these setting,  equation \eqref{wint-1st} can also be rewritten by
\begin{equation}\label{eq:Ryu-pre}
\begin{aligned}
&\int_{C_0[0,T]} f(x(t_1),\ldots, x(t_n)) d m_w(x)\\ 
&=\int_{\mathbb R^n}  f\bigg(\sqrt{t_1-t_0} w_1, 
\sqrt{t_1-t_0} w_1 +\sqrt{t_2-t_1} w_2,\ldots,\\
&\qquad\qquad\qquad\qquad\qquad\qquad\qquad\qquad
\sum_{j=1}^n \sqrt{t_{j}-t_{j-1}} w_j\bigg)
d \gamma_{G}(\vec w),
\end{aligned}
\end{equation}
 where $\gamma_G$ is the standard Gaussian measure given by \eqref{sgm}.

\begin{remark}
The classical Wiener space $C_0[0,T]$ with supremum norm 
can be considered as a (closed) subspace of the function space (topological product space)
\[
\mathbb R^{[0,T]}=\{x: x \mbox{ is a } \mathbb R\mbox{-valued function on } [0,T]\}.
\]
Likewise, the Wiener paths space $\mathcal C_0(\mathbb B)$ also can be considered as  a subspace 
of the function space
\[
\mathbb B^{[0,T]}=\{\mathfrak x: \mathfrak x \mbox{ is a } \mathbb B\mbox{-valued function on } [0,T]\}.
\]
Thus, we can see that the general Wiener integration theorem, given by  \eqref{eq:Ryu},
for measurable functionals $f$ on the Wiener paths space $\mathcal C_0(\mathbb B)$
is a natural  extension of \eqref{eq:Ryu-pre}.
\end{remark}


 \subsection{Change of variables theorem on the function space $C_{a,b}[0,T]$}
We shall discuss  a change of variables theorem, such as \eqref{eq:Ryu-pre},  on the function space $C_{a,b}[0,T]$. 

Given an $n$-tuple $(t_1,\ldots,t_n)$ of  time  moments with $0= t_0 < t_1< \cdots<t_n \le T$, 
let $F: C_{a,b}[0,T]\to \mathbb C$ be a tame function  given by
\[
F(x)=f(x(t_1), x(t_2),\ldots, x(t_n))
\]
where $f:\mathbb R^n\to \mathbb C$ is a Lebesgue measurable function.
We consider the following transformation
$T:\mathbb R^n \to\mathbb R^n$ given by 
\begin{equation}\label{eq:Tt-survey}
 T_{\vec{t}}(\vec w)
\equiv T_{\vec{t}}(w_1,\ldots, w_n) 
=\big(L_{\vec{t},1}(\vec w),L_{\vec{t},2}(\vec w), \ldots, L_{\vec{t},n} (\vec{w})\big)
\end{equation}
where
\begin{equation}\label{eq:Ls-survey}
L_{\vec{t},k}(\vec{ w})
= \sum\limits_{l=1}^k \sqrt{b(s_l)-b(s_{l-1})} (w_l -a(t_l)) +a(t_k)
\end{equation}
for each $k=1,\ldots,n$. 
Let $\vec u = T_{\vec{t}}(\vec w)$.
Then it follows that for each $j\in\{1,2,\ldots,n\}$,
 \[
\sqrt{b(t_j)-b(t_{j-1})}(w_j- a(t_j)) =(u_j-a(t_j))-( u_{j-1}-a(t_{j-1}))
\] 
or, equivalently,
\begin{equation}\label{eq:basic001-in-obervation}
(w_j- a(t_{j-1})) =\frac{(u_j-a(t_j))-  (u_{j-1}-a(t_{j-1}))}{\sqrt{b(t_{j-1})-b(t_{j-1})}}.
\end{equation} 
In this case, we see that
\begin{equation}\label{eq:jj-survey}
\mathcal J\bigg(\frac{u_1,\ldots, u_n}{w_1,\ldots, w_n}\bigg) 
 =\prod_{j=1}^n \sqrt{b(t_j)-b(t_{j-1})}.
\end{equation} 
Using \eqref{eq:basic001}, \eqref{eq:basic001-in-obervation}
and \eqref{eq:jj-survey}, it follows that
\[
\begin{aligned}
&\int_{C_{a,b}[0,T]} f(x(t_1),\ldots, x(t_n)) d\mu(x)\\ 
&=\int_{\mathbb R^n}  f(u_1,\ldots, u_n)
K_{n}(\vec t ,\vec u)d \vec u\\
&=\int_{\mathbb R^n}  f(u_1,\ldots, u_n)
\bigg(\prod_{j=1}^n2\pi(b(t_j)-b(t_{j-1}))\bigg)^{-1/2}\\
&\qquad\qquad \times
\exp\bigg\{-\frac{1}{2}\sum_{j=1}^n \frac{[(u_j-a(t_j))-(u_{j-1}-a(t_{j-1}))]^2}{b(t_j)-b(t_{j-1})}\bigg\}d \vec u\\
&=\int_{\mathbb R^n}  f (T_{\vec t}(w_1,\ldots, w_n) ) \exp\bigg\{-\frac{1}{2}\sum_{j=1}^n(w_j- a(t_{j-1}))^2\bigg\} \\
&\qquad\qquad\times 
\bigg(\prod_{j=1}^n 2\pi(b(t_j)-b(t_{j-1}))\bigg)^{-1/2}  
\bigg|\mathcal J\bigg(\frac{u_1,\ldots, u_n}{w_1,\ldots, w_n}\bigg) \bigg|d \vec w\\
&= (2\pi)^{-n/2} \int_{\mathbb R^n}   f (T_{\vec t}(w_1,\ldots, w_n) )  
 \exp\bigg\{-\frac{1}{2}\sum_{j=1}^n(w_j- a(t_{j-1}))^2\bigg\} d \vec w\\
\end{aligned}
\]
\[
\begin{aligned}
&= \int_{\mathbb R^n}   f (T_{\vec t}(w_1,\ldots, w_n) )    d \gamma_{G}^{a;\vec t}(\vec w), 
\end{aligned}
\]
 where $\gamma_G^{a;\vec t}$ is the Gaussian measure on $\mathbb R^n$ (with mean vector $(a(t_1),a(t_2),$$\cdots,$ $a(t_n))$) 
 given by
\[
 \gamma_{G}^{a;\vec t}(B)=(2\pi)^{-n/2} \int_{B} \exp\bigg\{-\frac{1}{2}\sum_{j=1}^n(w_j- a(t_{j-1}))^2\bigg\} d\vec u 
\]
for $B\in \mathbb R^n$.

In view of these observations  we will study  a  construction of the paths space $\mathcal C_0^{\mathrm{gBm}}(C_{a,b}[0,T])$,
such as the Wiener paths space $\mathcal C_0(\mathbb B)$.
The general  paths space $\mathcal C_0^{\mathrm{gBm}}(C_{a,b}[0,T])$ is associated with the GBMP 
determined by the continuous functions $a(\cdot)$ and $b(\cdot)$ on $[0,T]$.

\setcounter{equation}{0}
\section{The paths space $\mathcal C_0$}

\par
Let $\mathcal C_0^{\mathrm{gBm}} \equiv \mathcal C_0^{\mathrm{gBm}}(C_{a,b}[0,T])$ be the class of all $C_{a,b}[0,T]$-valued 
continuous functions $\mathfrak x$  on the compact interval  $[0,T]$ with $\mathfrak x(0)=0$.
From \cite{KLe73} it follows that the class $\mathcal C_0^{\mathrm{gBm}}$ is a real separable Banach space 
with the norm 
\[
\|\mathfrak x\|_{\mathcal C_0^{\mathrm{gBm}}}=\sup_{s \in [0,T]} \|\mathfrak x(s)\|_{C_{a,b}[0,T]}
\]
and the minimal $\sigma$-field making the mapping $\mathfrak x \rightarrow \mathfrak x(s)$ 
measurable is  the Borel $\sigma$-field $\mathcal B(\mathcal C_0^{\mathrm{gBm}})$ on $\mathcal C_0^{\mathrm{gBm}}$.

\begin{remark}
The  paths space $(\mathcal C_0^{\mathrm{gBm}}, \|\cdot\|_{\mathcal C_0^{\mathrm{gBm}}})$ can be 
considered as a subspace of the topological product space $(C_{a,b}[0,T])^{[0,T]}$.
\end{remark} 

Furthermore, the generalized Brownian motion process in $C_{a,b}[0,T]$ induces a probability 
measure $ {\mu}_{\mathcal C_0^{\mathrm{gBm}}}$ on $(\mathcal C_0^{\mathrm{gBm}}, \mathcal{W}(\mathcal C_0^{\mathrm{gBm}}))$ 
where $\mathcal{W}(\mathcal C_0^{\mathrm{gBm}})$ is the complete $\sigma$-field in the sense of Carath\'eodory 
extension on the Borel $\sigma$-field $\mathcal{B}(\mathcal C_0^{\mathrm{gBm}})$. We will  introduce a concrete 
form of ${\mu}_{\mathcal C_0^{\mathrm{gBm}}}$. Let $\vec{s}=(s_1,\cdots,s_n)$ be given  with 
$0=s_0<s_1<\cdots< s_n\le T$, and let $T_{\vec{s}} : C_{a,b}^n[0,T] \rightarrow C_{a,b}^n[0,T]$ 
be defined by
\begin{equation}\label{eq:Tt}
T_{\vec{s}}(\vec{x})
\equiv T_{\vec{s}}(x_1,\ldots, x_n) 
=\big(L_{\vec{s},1}(\vec{x}),L_{\vec{s},2}(\vec{x}), \ldots, L_{\vec{s},n} (\vec{x})\big)
\end{equation}
where
\begin{equation}\label{eq:Ls}
L_{\vec{s},k}(\vec{x})
= \sum\limits_{l=1}^k \sqrt{b(s_l)-b(s_{l-1})} (x_l -a) +a
\end{equation}
for each $k=1,\ldots,n$. 
Let $\mu^n= \times_{1}^n \mu=\underbrace{\mu \times \cdots \times \mu}_{n-\text{times}}$ be the 
product measure on the product function space $C_{a,b}^n[0,T]$.  We then define a set function 
$\mu_{\vec{s}}$ on $\mathcal{B}(C_{a,b}^n[0,T])$ by 
\begin{equation}\label{projection-c0-A}
\mu_{\vec s}(E)=   \mu^n (T_{\vec s}^{-1}(E)) 
\end{equation}
for every $E$ in $\mathcal{B}(C_{a,b}^n[0,T])$. 
Then $\mu_{\vec s}$ is a Borel measure.
Next let $\mathfrak P_{\vec{s}} : \mathcal C_0^{\mathrm{gBm}} \to C_{a,b}^n[0,T]$ 
be the function with 
\begin{equation}\label{projection-c0-B}
\mathfrak P_{\vec{s}}(\mathfrak x)
= (\mathfrak  x(s_1),\mathfrak  x(s_2),\cdots,\mathfrak  x(s_n)).
\end{equation}
For Borel subsets $B_1,B_2,\cdots,B_n$ in $\mathcal B(C_{a,b}[0,T])$, 
$\mathfrak P_{\vec{s}}^{-1}(\prod_{l=1}^n B_l)$ is called a cylinder set with respect to 
$B_1,B_2,\cdots,B_n$.  
For each positive integer $n$  and  the cylinder function $\mathfrak P_{\vec s}$
given by \eqref{projection-c0-B} with $\vec s=(s_1,\ldots,s_n)$, let
\[
\mathcal I_{\vec s}=\bigg\{\mathfrak P_{\vec s}^{-1}\bigg(\prod_{j=1}^n B_j\bigg):  
B_1,B_2,\ldots,B_n  \in  \mathcal B(C_{a,b}[0,T]) \bigg\}
\]
and let
 $\mathcal{I}=\cup \mathcal I_{\vec s}$ where the union is over all ordered  multidimensional tuples $\vec s$.
Then for each $G\in \mathcal I$, there is a multidimensional tuple $\vec s=(s_1,\ldots, s_n)$ with 
 $0=s_0 < s_1 < \cdots < s_n \leq T$ such that $G\in \mathcal I_{\vec s}$.  
Given any positive integer $n$  and  Borel subsets $B_1,B_2,\ldots,B_n$ in $\mathcal B(C_{a,b}[0,T])$,
we  define  a set function  $\mu_{\mathcal{C}_0^{\mathrm{gBm}}}$ on $\mathcal{I}$ by 
\begin{equation}\label{projection-c0-C}
{\mu}_{\mathcal C_0^{\mathrm{gBm}}}\bigg(\mathfrak  P_{\vec{s}}^{-1}
\bigg(\prod_{l=1}^n B_l\bigg)\bigg)=\mu_{\vec{s}}\bigg(\prod_{l=1}^n B_l\bigg).
\end{equation}
Then ${\mu}_{\mathcal C_0}^{\mathrm{gBm}}$ is well-defined and is countably additive on $\mathcal{I}$. 
Using the Carath\'eodory extension process, it can be extended on the  $\sigma$-field  
$\mathcal{W}(\mathcal C_0^{\mathrm{gBm}})$, where $\mathcal{W}(\mathcal C_0^{\mathrm{gBm}})$ is the completion of the 
Borel  $\sigma$-field $\mathcal{B}(\mathcal C_0^{\mathrm{gBm}})$. The extended measure on  
$\mathcal{W}(\mathcal C_0^{\mathrm{gBm}})$  will be again denoted by  ${\mu}_{\mathcal C_0^{\mathrm{gBm}}}$. Hence we 
have the measure space $(\mathcal C_0^{\mathrm{gBm}},\mathcal{W}(\mathcal  C_0^{\mathrm{gBm}}),{\mu}_{\mathcal C_0^{\mathrm{gBm}}})$. 
This measure space is called the  space of generalized Wiener paths.
 
\begin{remark}
The transform $T_{\vec s}$ given by \eqref{eq:Tt} is formulated based 
on the transform   \eqref{eq:Tt-survey} together with  the probability low of 
the cylinder function \eqref{projection-c0-B}.  
\end{remark}
 
\par
Applying the techniques similar to those used in  the proof of the 
Kolmogorov extension theorem \cite[pp.4--17]{Yeh73},
we obtain the  following two lemmas.

\begin{lemma}
For each  multidimensional tuple  $\vec s$
with  $0=s_0 < s_1 < \cdots < s_n \leq T$, 
the class  $\mathcal I_{\vec s}$ is a $\sigma$-field. 
Furthermore, the class $\mathcal{I}=\cup \mathcal I_{\vec s}$  is a  field of subsets of $\mathcal C_0^{\mathrm{gBm}}$. 
\end{lemma}

\begin{lemma}
The set function  $\mu_{\mathcal C_0^{\mathrm{gBm}}}$ is well-defined and is countably additive 
on the field  $\mathcal{I}$.
Furthermore,  $\mu_{\mathcal C_0^{\mathrm{gBm}}}$ can be extended uniquely to be a probability 
measure  on the   $\sigma$-field  $\sigma(\mathcal I)$ generated by $\mathcal I$. 
\end{lemma}

\begin{remark} \label{rem:borel-cylinder}
The   $\sigma$-field  $\sigma(\mathcal I)$ generated by $\mathcal I$
coincides with the Borel $\sigma$-field $\mathcal B(\mathcal C_0^{\mathrm{gBm}})$.
Thus $\mathcal W(\mathcal C_0^{\mathrm{gBm}})$ is the completion of $\sigma(\mathcal I)$ .
\end{remark}

\par
Using the Carath\'eodory extension process, 
we also obtain the  following lemma.

\begin{lemma}
The measure  $\mu_{\mathcal C_0^{\mathrm{gBm}}}$ can be extended  uniquely
on the complete  $\sigma$-field  $\mathcal{W}(\mathcal C_0^{\mathrm{gBm}})$. 
\end{lemma}

The extended measure on  $\mathcal{W}(\mathcal C_0^{\mathrm{gBm}})$  will be again denoted by  
$\mu_{\mathcal C_0^{\mathrm{gBm}}}$. Hence we have the complete measure space 
$(\mathcal C_0^{\mathrm{gBm}},\mathcal{W}(\mathcal  C_0^{\mathrm{gBm}}),\mu_{\mathcal C_0^{\mathrm{gBm}}})$. 
This measure space is called the  space of GBMP paths 
(henceforth, GBMP paths space or paths space).

\par
Now, we introduce a paths space integration theorem on the paths space $\mathcal C_0^{\mathrm{gBm}}$.

\begin{theorem}[Paths Space Integration Theorem] \label{thm:well} 
Let $\vec{s}=(s_1,\cdots,s_n)$ be given with $0=s_0 < s_1 < \cdots < s_n \le T$ 
and let $f: C_{a,b}^n[0,T] \rightarrow \mathbb{C}$ be a $\mathcal W(C_{a,b}^n[0,T])$-measurable 
function. Then
\begin{equation}\label{eq:well}
\int_{\mathcal C_0} f(\mathfrak x(s_1),\ldots,\mathfrak x (s_n)) 
d{\mu}_{\mathcal C_0^{\mathrm{gBm}}}(\mathfrak x)  
\stackrel{*}{=} 
\int_{C_{a,b}^n[0,T]} f( T_{\vec{s}}(x_1,\ldots,x_n)) d \mu^n(x_1,\ldots,x_n),
\end{equation}
where $\stackrel{*}{=}$ means that if either side exists, both sides exist and equality holds.
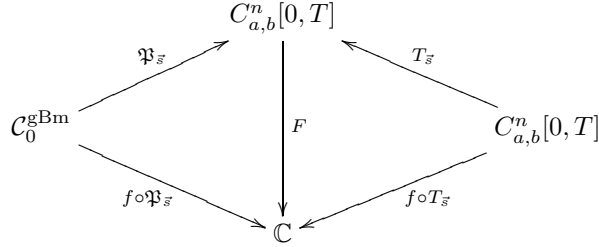
\[
\begin{array}{c}
\xymatrix{ && C_{a,b}^n[0,T] \ar[dd]^{F}  & &\\
\mathcal C_0^{\mathrm{gBm}} \ar[rru]^{ \mathfrak P_{\vec s}} \ar[rrd]_{f \circ \mathfrak P_{\vec s}}   
&&&& C_{a,b}^n[0,T]\ar[lld]^{ f\circ T_{\vec s}} \ar[llu]_{ T_{\vec s}}\\
&&\mathbb C &&\\
}
\end{array}
\] 
\end{theorem}

\section{Proof of the paths space integration theorem}
\par
In order to prove the paths space integration theorem,
we need the following  lemmas.

\begin{lemma}[\cite{halmos}]
Let $T$ be a  measurable transform from a measure space $(X,\mathcal S,\mu)$
into a measurable space $(Y,\mathcal T)$, and let $g$ be an extended real valued 
measurable function on $Y$. Then 
\begin{equation}\label{halmos}
\int_{Y} g(y)d \mu\circ T^{-1}(y)=\int_X g(T(x))d\mu(x)
\end{equation}
in the sense that if either integral exists, 
then both sides exist and they are equal.
\end{lemma}

\begin{lemma}\label{lem:1-2}  Let $\vec s=(s_1,\ldots, s_n)$ 
be a multidimensional tuple with  $0=s_0 < s_1 < \cdots < s_n \leq T$. 
Then,   
\begin{itemize}
  \item[(i)] $\mathfrak P_{\vec s}^{-1}(B)$ is in $\mathcal B(\mathcal C_0^{\mathrm{gBm}})$ for every $B\in \mathcal B(C_{a,b}^n[0,T])$.
  \item[(ii)] For a subset $B$ of $C_{a,b}^n[0,T]$ with
$\mathfrak P_{\vec s}^{-1}(B)\in  \mathcal B(\mathcal C_0^{\mathrm{gBm}})$,   $B$ is in $\mathcal B(C_{a,b}^n[0,T])$.
\end{itemize}
\end{lemma} 
\begin{proof} (i) The cylinder function  $\mathfrak P_{\vec s}$ given by \eqref{projection-c0-B} is 
continuous, it is $\mathcal B(\mathcal C_0^{\mathrm{gBm}})-\mathcal B(C_{a,b}^n[0,T])$-measurable.
Hence $\mathfrak P_{\vec s}^{-1}(B) \in \mathcal B(\mathcal C_0^{\mathrm{gBm}})$  for any  $B\in \mathcal B(C_{a,b}^n[0,T])$.

(ii) Given a multidimensional tuple $\vec s=(s_1,\ldots, s_n)$ 
with  $0=s_0 < s_1 < \cdots < s_n \leq T$, define a map 
$H_{\vec s}: C_{a,b}^n[0,T] \to \mathcal C_0^{\mathrm{gBm}}$ by 
$H_{\vec s}(\vec x)\equiv H_{\vec s}(x_1,\ldots,x_n)$ 
to be the polygonal path in $\mathcal C_0^{\mathrm{gBm}}$ 
such as 
\[
\begin{aligned}
&H_{\vec s}(x_1,\ldots,x_n)(s)\\
&=\begin{cases}
 x_{j-1}+\frac{b(s)-b(s_{j-1})}{b(s_j)-b(s_{j-1})}(x_j-x_{j-1}),&    s\in [s_{j-1},s_j], \,\,j=1,\ldots, n \\
x_n, & s\in[s_n,T]
\end{cases}, 
\end{aligned}
\]
where $x_0=0$ (the zero function on [0,T]).

\par
 We   note that the Borel $\sigma$-field 
$\mathcal B(\mathcal C_0^{\mathrm{gBm}})$ (resp. $\mathcal B(C_{a,b}[0,T])$)
can be generated   by the uniform topology   
induced by the sup norm $\|\cdot\|_{\mathcal C_0^{\mathrm{gBm}}}$ (resp. $\|\cdot\|_{C_{a,b}[0,T]}$).
Given $\vec x=(x_1,\ldots,x_n)\in C_{a,b}^n[0,T]$, let $\langle{\vec{x_k}}\rangle
=\langle{(x_{k,1},\ldots,x_{k,n})}\rangle$ be a sequence in $C_{a,b}^n[0,T]$
which converges to $\vec x$, i.e.,  for each $j=1,\ldots,n$,  $\lim_{k\to\infty}x_{k,j}=x_j$ in 
$(C_{a,b}[0,T], \|\cdot\|_{C_{a,b}[0,T]})$.
From this and the definition of $H_{\vec s}$,  it follows  that $H_{\vec s}(\vec{x_{k}})$ 
converges to $H_{\vec s}(\vec x)$, uniformly on $[0,T]$, as $k\to\infty$.
Hence  the map $H_{\vec s}: (C_{a,b}^n[0,T],\mathcal B(C_{a,b}^n[0,T])) \to  
(\mathcal C_0^{\mathrm{gBm}},\mathcal B(\mathcal C_0^{\mathrm{gBm}}))$ is continuous,
and so is $\mathcal B(C_{a,b}^n[0,T])-\mathcal B(\mathcal C_0^{\mathrm{gBm}})$-measurable.
It thus follows that 
$H_{\vec s}^{-1}(\mathfrak P_{\vec s}^{-1}(B)) \in \mathcal B(C_{a,b}^n[0,T])$
for each $B\in \mathcal B(C_{a,b}^n[0,T])$.
 To complete the proof of  the assertion (ii), it thus   suffices to show that
for each $B\in \mathcal B(C_{a,b}^n[0,T])$, $H_{\vec s}^{-1}(\mathfrak P_{\vec s}^{-1}(B))=B$. 
First, take any $(x_1,\ldots,x_n)\in H_{\vec s}^{-1}(\mathfrak P_{\vec s}^{-1}(B))$.
Then, $H_{\vec s}(x_1,\ldots,x_n)\in \mathfrak P_{\vec s}^{-1}(B)$ and hence
$ \mathfrak P_{\vec s}(H_{\vec s}(x_1,\ldots,x_n))\in B$.
Thus, it  follows that
\[
\begin{aligned}
(x_1,\ldots,x_n) 
&=(H_{\vec s}(x_1,\ldots,x_n)(s_1), \ldots, H_{\vec s}(x_1,\ldots,x_n)(s_n))\\
&=\mathfrak P_{\vec s}(H_{\vec s}(x_1,\ldots,x_n)) \\
&\in B. 
\end{aligned} 
\]
Conversely,  by the inverse image property of maps, it follows that for each $B\in \mathcal B(C_{a,b}^n[0,T])$, 
$B \subset (\mathfrak P_{\vec s} \circ H_{\vec s})^{-1}(B)= H_{\vec s}^{-1}(\mathfrak P_{\vec s}^{-1}(B))$,
as desired.
\end{proof}

\begin{remark} \label{rem:2019-001}
Lemma \ref{lem:1-2} tells us that given any multidimensional tuple  $\vec s=(s_1,\ldots, s_n)$  
with  $0=s_0 < s_1 < \cdots < s_n \leq T$  and any  a subset $B$ of $C_{a,b}^n[0,T]$,
$\mathfrak P_{\vec t}^{-1} (B)\in \mathcal B(\mathcal C_0^{\mathrm{gBm}})$ if and only if $B\in \mathcal B(C_{a,b}^n[0,T])$.
 \end{remark}

 Our next theorem  follows quite readily from the techniques developed in
\cite[Section 3]{CA83} and Remark  \ref{rem:2019-001}.

\begin{lemma}[Converse measurability theorem]\label{thm:400}  
 Let $\vec s=(s_1,\ldots, s_n)$ be as in Lemma \ref{lem:1-2}.
For a subset $B$ of $C_{a,b}^n[0,T]$ with
$\mathfrak P_{\vec s}^{-1}(B)\in  \mathcal W(\mathcal C_{a,b}^{\mathrm{gBm}})$,   $B$ is in $\mathcal W(C_{a,b}^n[0,T])$.
\end{lemma}

In view of equation \eqref{projection-c0-C}, it follows the following corollaries.

\begin{corollary}\label{coro:2019-002}
Given any multidimensional tuple  $\vec s=(s_1,\ldots, s_n)$  
with  $0=s_0 < s_1 < \cdots < s_n \leq T$  and any  subset $B$ of $C_{a,b}^n[0,T]$,
he following assertions are equivalent.
\begin{itemize}
  \item[(i)] $B$ is   in $\mathcal W(C_{a,b}^n[0,T])$; and
  \item[(iii)]  $\mathfrak P_{\vec t}^{-1} (B)$ is   in $\mathcal W( \mathcal C_0^{\mathrm{gBm}})$.  
\end{itemize}
\end{corollary}

\begin{corollary}\label{coro:2019-003}
Given any multidimensional tuple  $\vec s=(s_1,\ldots, s_n)$  
with  $0=s_0 < s_1 < \cdots < s_n \leq T$  and any  subset $B$ of $C_{a,b}^n[0,T]$,
he following assertions are equivalent.
\begin{itemize}
  \item[(i)] $B$ is a $\mu^n$-null set in $\mathcal W(C_{a,b}^n[0,T])$;
  \item[(ii)] $B$ is a $\mu_{\vec s}$-null set in $\mathcal W(C_{a,b}^n[0,T])$; and
  \item[(iii)]  $\mathfrak P_{\vec t}^{-1} (B)$ is a $\mu_{\mathcal C_0^{\mathrm{gBm}}}$-null set in $\mathcal W( \mathcal C_0^{\mathrm{gBm}})$.  
\end{itemize}
\end{corollary}
 
\par
Our next lemma  follows quite readily from  Corollaries  \ref{coro:2019-002} and \ref{coro:2019-003}.
\begin{lemma}\label{lem:300}
 Let $\vec s=(s_1,\ldots, s_n)$ be as in Lemma \ref{lem:1-2}.
Then,   
$\mathfrak P_{\vec s}^{-1}(B)$ is in $\mathcal W(\mathcal C_0^{\mathrm{gBm}})$ for every $B\in \mathcal M(C_{a,b}^n[0,T])$.
 In other words the  cylinder function $\mathfrak P_{\vec s}$ given by
\eqref{projection-c0-B}  is $\mathcal W(\mathcal C_0^{\mathrm{gBm}})-\mathcal M(C_{a,b}^n[0,T])$-measurable.
\end{lemma}
 
\par
We are now ready to present the proof  of Theorem \ref{thm:well}.

\begin{proof}[Proof of Theorem \ref{thm:well}]
 We may assume, without loss of generality, that  $F$ is a real-valued function. 
We first note that for any 
$\mathcal W(C_{a,b}^n[0,T])$-measurable function $F$ on $C_{a,b}^n[0,T]$,
\[
f(\mathfrak x(s_1),\ldots, \mathfrak x (s_n) )=f\circ \mathfrak P_{\vec s}(\mathfrak x).
\]
 Thus, by Lemma \ref{lem:300}, $f(\mathfrak x(s_1),\ldots, \mathfrak x (s_n) )$
is $\mathcal W (\mathcal C_0)$-measurable, as a function of $\mathfrak x$.
Next, using \eqref{projection-c0-B}, \eqref{halmos}, \eqref{projection-c0-C}, 
\eqref{projection-c0-A},  and  \eqref{halmos}  again, it follows that
\[
\begin{aligned}
&\int_{\mathcal C_0^{\mathrm{gBm}}} f(\mathfrak x(s_1),\ldots,\mathfrak x (t_n)) d{\mu}_{\mathcal C_0^{\mathrm{gBm}}}(\mathfrak x) \\ 
&=\int_{\mathcal C_0^{\mathrm{gBm}}} f(\mathfrak P_{\vec s}(\mathfrak x)) {\mu}_{\mathcal C_0^{\mathrm{gBm}}}(\mathfrak x) \\ 
&=\int_{C_{a,b}^n[0,T]} f(w_1,\ldots,w_n)
d  {\mu}_{\mathcal C_0^{\mathrm{gBm}}}\circ \mathfrak P_{\vec s}^{-1} (w_1,\ldots,w_n)\\ 
&=\int_{C_{a,b}^n[0,T]} f(w_1,\ldots,w_n)d  \mu_{\vec s} (\vec w)\\ 
&=\int_{C_{a,b}^n[0,T]} f(w_1,\ldots,w_n)d \mu^n \circ T_{\vec s}^{-1} (\vec w)\\ 
&=\int_{C_{a,b}^n[0,T]}  f(T_{\vec s}(x_1,\ldots, x_n)) d \mu^n(\vec x).
\end{aligned}
\]
This completes the proof of the theorem.
\end{proof}

\setcounter{equation}{0}
\section{Examples}

In this section we present interesting examples 
to which equation \eqref{eq:well} can be applied.
Our examples 
involves the PWZ stochastic integrals
$(w,\mathfrak x(s))^{\sim}$.
Thus, in this section, we have to guarantee the existence of the PWZ stochastic integral
$(w,\mathfrak x(s))^{\sim}$ for $\mathfrak x \in \mathcal C_0^{\mathrm{gBm}}$.
But, in view of Corollary \ref{coro:2019-003}, 
we obtain the following lemma.

\begin{lemma}\label{lem:pwz-frak}
For each $s\in(0,T]$ and  $w\in C_{a,b}'[0,T]$,  $(w,\mathfrak x(s))^{\sim}$  exists 
for $\mu_{\mathcal C_0^{\mathrm{gBm}}}$-a.s. $\mathfrak x\in \mathcal C_{0}^{\mathrm{gBm}}$. 
\end{lemma}
 
\par
We now ready to present several examples  
to which equation \eqref{eq:well} can be applied.

\begin{example}
We note that given a function $w$ in $C_{a,b}'[0,T]$, the PWZ stochastic integral 
$(w,\cdot)^{\sim}: C_{a,b}[0,T]\to\mathbb R$ is a Gaussian random variable 
with mean $(w,a)_{C_{a,b}'}$ and variance $\|w\|_{C_{a,b}'}^2$. Then using this fact,  
\eqref{eq:well} with $n=1$, and \eqref{eq:pwz-meam-type} with $x$ replaced with $a$, 
it follows that for each $s \in (0,T]$,
\[
\begin{aligned}
&\int_{\mathcal C_0^{\mathrm{gBm}}} (w, \mathfrak x(s))^{\sim} d {\mu}_{\mathcal C_0^{\mathrm{gBm}}}(\mathfrak x)\\
&=\int_{C_{a,b}[0,T]} (w,\sqrt{b(s)}(x-a)+a)^{\sim} d\mu(x)\\
&=\int_{C_{a,b}[0,T]}\big[ \sqrt{b(s)}(w,x)^{\sim}
-\sqrt{b(s)}(w,a)_{C_{a,b}'}+(w,a)_{C_{a,b}'}\big]d\mu(x)\\
&=\sqrt{b(s)}(w,a)_{C_{a,b}'}-\sqrt{b(s)}(w,a)_{C_{a,b}'}+(w,a)_{C_{a,b}'}\\
&= (w,a)_{C_{a,b}'} 
\end{aligned}
\]
and 
\[
\begin{aligned}
&\int_{\mathcal C_0^{\mathrm{gBm}}} [(w, \mathfrak x(s))^{\sim}]^2 d {\mu}_{\mathcal C_0^{\mathrm{gBm}}}(\mathfrak x)\\
&=\int_{C_{a,b}[0,T]}\Big\{ \sqrt{b(s)}\big[(w,x)^{\sim}
- (w,a)_{C_{a,b}'}\big]+(w,a)_{C_{a,b}'}\Big\}^2d\mu(x)\\
&=b(s)\int_{C_{a,b}[0,T]} [(w,x)^{\sim}- (w,a)_{C_{a,b}'}]^2 d\mu(x)\\
& \quad  
 +2\sqrt{b(s)}(w,a)_{C_{a,b}'}\int_{C_{a,b}[0,T]} 
\big[(w,x)^{\sim}- (w,a)_{C_{a,b}'}\big]  d\mu(x)
+(w,a)_{C_{a,b}'}^2\\
&=b(s)\|w\|_{C_{a,b}'}^2+ (w,a)_{C_{a,b}'}^2.
\end{aligned}
\]
\end{example}

\begin{example}
Let $w_1, w_2 \in C_{a,b}'[0,T]$ and let $s_1,s_2 \in (0,T]$ with $s_1<s_2$. 
Then using  \eqref{eq:well}  with $n=2$, the Fubini theorem, \eqref{eq:pwa-cov-semi}, 
and \eqref{eq:pwz-meam-type}, it follows that 
\[ 
\begin{aligned}
&\int_{\mathcal C_0^{\mathrm{gBm}}}  
(w_1, \mathfrak x(s_1))^{\sim}
(w_2, \mathfrak x(s_2))^{\sim} d {\mu}_{\mathcal C_0^{\mathrm{gBm}}}(\mathfrak x)\\
&=\int_{C_{a,b}^2[0,T]}\Big\{ \sqrt{b(s_1)}\big[(w_1,x_1)^{\sim}- (w_1,a)_{C_{a,b}'}\big]
+(w_1,a)_{C_{a,b}'}\Big\}\\
&\qquad\times
\Big\{ \sqrt{b(s_1)}\big[(w_2,x_1)^{\sim}- (w_2,a)_{C_{a,b}'}\big]\\
&\qquad \qquad
+ \sqrt{b(s_2)-b(s_1)}\big[(w_2,x_2)^{\sim}- (w_2,a)_{C_{a,b}'}\big] 
+(w_2,a)_{C_{a,b}'}\Big\}d\mu^2(x_1,x_2)\\
&=b(s_1)(w_1,w_2)_{C_{a,b}'} + (w_1,a)_{C_{a,b}'}(w_2,a)_{C_{a,b}'}.
\end{aligned}
\] 
In particular, taking $s_1,s_2\in(0,T]$ and $t,t_1,t_2\in [0,T]$,
and using  \eqref{kernel-repre} and \eqref{rkernel}, we obtain that
\[ 
\int_{\mathcal C_0^{\mathrm{gBm}}}  \mathfrak x(s_1)(t) \mathfrak x(s_2)(t)
d {\mu}_{\mathcal C_0^{\mathrm{gBm}}}(\mathfrak x)
=\min\{b(s_1),b(s_2)\}b(t) + a^2(t)
\] 
and
\[ 
\begin{aligned}
&\int_{\mathcal C_0^{\mathrm{gBm}}}  
\mathfrak x(s_1)(t_1) \mathfrak x(s_2)(t_2)d {\mu}_{\mathcal C_0^{\mathrm{gBm}}}(\mathfrak x)\\
&=\min\{b(s_1),b(s_2)\}\min\{b(t_1),b(t_2)\} + a(t_1)a(t_2).
\end{aligned}
\]
\end{example}

\begin{example}
Let  $s \in (0,T]$ be fixed.
Using equation \eqref{eq:well}  with $n=1$ and  \eqref{eq:int-formula-cab}, 
it follows that given any  nonzero real number $\rho$ and  a function $w$ 
in $C_{a,b}'[0,T]$,
\[
\begin{aligned}
&\int_{\mathcal C_0^{\mathrm{gBm}}} \exp\{i\rho(w, \mathfrak x(s))^{\sim}  \} 
d {\mu}_{\mathcal C_0^{\mathrm{gBm}}}(\mathfrak x) \\
&=\int_{C_{a,b}[0,T]}\exp\big\{i\rho \big(w, \sqrt{b(s)}(x-a)+a\big)^{\sim}\big\} 
d\mu(x)\\
&=\int_{C_{a,b}[0,T]}\exp\big\{i\rho\sqrt{b(s)}(w,x)^{\sim}\big\}d\mu(x)
\exp\big\{-i\rho\sqrt{b(s)}(w,a)_{C_{a,b}'}+i\rho(w,a)_{C_{a,b}'}\big\} \\
&=\exp\bigg\{-\frac{1}{2}\rho^2 b(s)\|w\|_{C_{a,b}'}^2+i\rho(w,a)_{C_{a,b}'}\bigg\}.
\end{aligned}
\]
\end{example}

\begin{example}
Let  $s_1, s_2  \in (0,T]$ be given with $s_1<s_2$.
Using \eqref{eq:well} and the Fubini theorem, and applying \eqref{eq:int-formula-cab},
it follows that given any  nonzero real number $\rho$ and any  functions $w_1$ and $w_2$
in $C_{a,b}'[0,T]$,
\[
\begin{aligned}
&\int_{\mathcal C_0^{\mathrm{gBm}}} \exp\big\{i\rho\big(w_1,\mathfrak x(s_1)\big)^{\sim} 
  + i\rho\big(w_2,\mathfrak x(s_2)\big)^{\sim} \big\} d {\mu}_{\mathcal C_0^{\mathrm{gBm}}}(x)\\
&=\int_{C_{a,b}^2[0,T]} \exp\Big\{i\rho\Big(w_1,\sqrt{b(s_1)}(x_1-a)+a\Big)^{\sim} \\
&\qquad  
+i\rho \Big(w_2, \sqrt{b(s_1)}(x_1-a)+ \sqrt{b(s_2)-b(s_1)}(x_2-a)+a\Big)^{\sim}\Big\} 
d(\mu \times \mu)(x_1, x_2)\\
&=\int_{C_{a,b} [0,T]} \exp\Big\{i\rho\sqrt{b(s_1)}\big(w_1+w_2,x_1\big)^{\sim}\Big\}d\mu(x_1) \\
&\quad \times
\int_{C_{a,b} [0,T]} \exp\Big\{i\rho\sqrt{b(s_2)-b(s_1)}\big(w_2,x_2\big)^{\sim}\Big\}d \mu(x_2)\\
&\quad\times
\exp\Big\{-i\rho\sqrt{b(s_1)}(w_1+w_2,a)_{C_{a,b}'}  + i\rho(w_1,a)_{C_{a,b}'}\\
&\qquad\qquad\,\,\,
          -i\rho\sqrt{b(s_2)-b(s_1)}(w_2,a)_{C_{a,b}'}+i\rho(w_2,a)_{C_{a,b}'}\Big\}\\
&=\exp\bigg\{ -\frac{b(s_1)\rho^2}{2}\|w_1+w_2\|^2_{C_{a,b}'}+i\rho(w_1,a)_{C_{a,b}'}\\
&\qquad\qquad
-\frac{(b(s_2)-b(s_1))\rho^2}{2}\|w_2\|^2_{C_{a,b}'} +i\rho(w_2,a)_{C_{a,b}'} \bigg\}.
\end{aligned}
\]

\par
By an induction argument, it follows that
given any $n$-tuple $\vec s=(s_1,\ldots,s_n)$ with 
$0=s_0<s_1<\cdots <s_n \le T$ and any set $\{w_1,\ldots,w_n\}$
of functions in $C_{a,b}'[0,T]$, 
\[
\begin{aligned}
&\int_{\mathcal C_0^{\mathrm{gBm}}} \exp\bigg\{i\rho\sum_{k=1}^n
\big(w_k,\mathfrak x(s_k)\big)^{\sim} \bigg\} d {\mu}_{\mathcal C_0^{\mathrm{gBm}}}(x)\\
&=\exp\bigg\{ -\rho^2\sum_{k=1}^n\frac{(b(s_k)-b(s_{k-1})}{2}
\bigg\|\sum_{l=k}^nw_l\bigg\|^2_{C_{a,b}'}
+i\rho\sum_{k=1}^n(w_k,a)_{C_{a,b}'} \bigg\}.
\end{aligned}
\]
\end{example}

\setcounter{equation}{0}
\section{Analytic  paths space  Feynman integral}

As an application of the paths space integral, we suggest an
 analytic paths  space Feynman integral for functionals $F$ on $\mathcal C_0^{\mathrm{gBm}}$. 
In this section, we give a class of certain bounded cylinder functionals
whose analytic paths  space integral and  analytic paths  space Feynman integral 
on $\mathcal C_0^{\mathrm{gBm}}$ exist.

\par
A subset $B$ of $\mathcal C_0^{\mathrm{gBm}}$ is said to be scale-invariant measurable provided 
$\rho B$ is $\mathcal{W}(\mathcal C_0^{\mathrm{gBm}})$-measurable for all $\rho>0$, and a scale-invariant 
measurable set $N$ is said to be a scale-invariant null set provided $\mu_{\mathcal C_0^{\mathrm{gBm}}}(\rho N)=0$ 
for all $\rho>0$. A property that holds except on a scale-invariant null set is  said to hold 
scale-invariant almost everywhere (s-a.e.). A functional $F$ is said to be scale-invariant 
measurable provided $F$ is defined on a scale-invariant measurable set and $F(\rho\,\,\cdot\,)$ 
is $\mathcal{W}(\mathcal C_0^{\mathrm{gBm}})$-measurable for every $\rho>0$. 

\par
Throughout rest of this paper, for each $\lambda \in  \mathbb {\widetilde C}_+$, 
$\lambda^{-1/2}$ is always chosen to have positive real part.

\begin{definition}
Let $F:\mathcal C_{0}^{\mathrm{gBm}} \rightarrow \mathbb C $ be a scale-invariant  measurable functional 
such that   the paths space integral 
\[
J(\lambda)= \int_{\mathcal C_0^{\mathrm{gBm}}} F(\lambda^{-1/2}\mathfrak x)d{\mu}_{\mathcal C_0^{\mathrm{gBm}}}(\mathfrak x)
\]
exists as a finite number for all $\lambda>0$. 
If there exists a function $J^*(\lambda)$  analytic in $\mathbb  C_+$ such that  
$J^*(\lambda)=J(\lambda)$  for all $\lambda>0$,  then $J^*(\lambda)$ is defined to be the  
analytic paths space  integral of $F$ over $\mathcal C_0^{\mathrm{gBm}}$ with parameter $\lambda$, and 
for $\lambda \in \mathbb C_+$ we write 
\[
\int_{\mathcal C_0^{\mathrm{gBm}}}^{\mathrm{an}_\lambda} 
F(\mathfrak x) d\mu_{\mathcal C_0^{\mathrm{gBm}}}(\mathfrak  x) 
=J^*(\lambda).
\]
Let $q \ne 0$ be a real number and let $F$ be a functional such that 
$\int_{\mathcal C_0^{\mathrm{gBm}}}^{\mathrm{an}_\lambda} 
F(\mathfrak x) d\mu_{\mathcal C_0^{\mathrm{gBm}}}(\mathfrak  x)$ 
exists for all $\lambda \in \mathbb C_+$. If the following limit exists, we call it the analytic 
paths space Feynman integral of $F$ with parameter $q$ and we write 
\begin{equation}\label{eq:def-Fint}
\int_{\mathcal C_0^{\mathrm{gBm}}}^{\mathrm{anf}_q} 
F(\mathfrak x) d\mu_{\mathcal C_0^{\mathrm{gBm}}}(\mathfrak  x) 
=\lim_{\lambda \rightarrow -iq}	
\int_{\mathcal C_0^{\mathrm{gBm}}}^{\mathrm{an}_\lambda} F(\mathfrak x) d\mu_{\mathcal C_0^{\mathrm{gBm}}}(\mathfrak  x) 
\end{equation}
where $\lambda$ approaches $-iq$ through values in $\mathbb  C_+$.
\end{definition}

\subsection{Cylinder functionals in $\widehat{\mathfrak T}_{G_m,s}$}\label{sec:cylinder} 

\par 
A functional $F$ is called a cylinder functional on $\mathcal C_0^{\mathrm{gBm}}$ if there exists a finite 
subset $\{h_1, \ldots, h_k\}$ of $C_{a,b}'[0,T]$ such that
\begin{equation}\label{eq:cy}
F(\mathfrak x)
=\psi(( h_1, \mathfrak x(s))^{\sim}, \ldots, ( h_k ,\mathfrak x(s))^{\sim}), 
\quad \mathfrak x\in \mathcal C_0^{\mathrm{gBm}} 
\end{equation}
where $\psi$ is a complex-valued Borel measurable function on $\mathbb R^k$. It is easy to 
show that for given cylinder functional $F$  of the form \eqref{eq:cy} there exists an orthonormal  
subset $\{g_1, \ldots, g_m\}$ of $C_{a,b}'[0,T]$  such that  $F$ is expressed as 
\begin{equation}\label{eq:cylinder}
F(\mathfrak x)=f(( g_1,\mathfrak x(s))^{\sim}, 
\ldots,  ( g_m,\mathfrak x(s))^{\sim}), \quad \mathfrak x\in \mathcal C_0^{\mathrm{gBm}}  
\end{equation}
where $f$ is a complex Borel measurable function on $\mathbb R^m$. Thus we lose no generality   
in assuming that every cylinder functional on $C_{a,b}[0,T]$ is of the form \eqref{eq:cylinder}.

\begin{definition}
Let $\mathcal M (\mathbb R^m)$ denote the space of complex-valued Borel measures 
on $\mathcal B (\mathbb R^m)$. It is well known that a complex-valued Borel measure 
$\nu$ necessarily has a finite total variation $\|\nu\|$, and $\mathcal M (\mathbb R^m)$ 
is a Banach algebra under the norm $\|\cdot\|$  and with convolution as multiplication.

\par  
For $\nu\in \mathcal M (\mathbb R^m)$, the Fourier transform $\widehat \nu$ of $\nu$
is a complex-valued function defined on $\mathbb R^m$ by the formula
\begin{equation}\label{eq:hatnu}
\widehat\nu (\vec u)=\int_{\mathbb R^m} \exp\bigg\{ i \sum\limits_{j=1}^m u_j v_j \bigg\}
d \nu (\vec v) 
\end{equation}
where $\vec u=(u_1,\ldots,u_m)$ and $\vec v=(v_1,\ldots,v_m)$ are in $\mathbb R^m$.
\end{definition}

\par 
Let $G_m=\{ g_1, \ldots, g_m\}$ be an orthonormal subset of $C_{a,b}'[0,T]$. 
Given $s\in (0,T]$, define a functional $F: \mathcal C_{0}^{\mathrm{gBm}} \to \mathbb C$ by
\begin{equation}\label{Fhatnu}
F(\mathfrak x)
=\widehat \nu(( g_1 ,\mathfrak x(s) )^{\sim},\ldots, ( g_m,\mathfrak x(s))^{\sim}), 
\quad \mathfrak  x \in \mathcal C_0^{\mathrm{gBm}} 
\end{equation}
where $\widehat \nu$ is the Fourier transform of  $\nu$ in $\mathcal M (\mathbb R^m)$. 
Then $F$ is a bounded cylinder functional since  $|\widehat \nu (\vec u)|\le \|\nu\| <+\infty$. 
Let $\widehat{\mathfrak T}_{G_m,s}$ be the space of all functionals $F$ on $\mathcal C_0^{\mathrm{gBm}}$ 
having the form \eqref{Fhatnu}. Note that $F \in\widehat{\mathfrak T}_{G_m,s}$ implies 
that $F$ is scale-invariant measurable on $\mathcal C_0^{\mathrm{gBm}}$.

\par
We first show that the analytic paths space  integral  of the functional  $F$ given 
by equation \eqref{Fhatnu} exists.

\begin{theorem}\label{thm:analytic} 
Let $F\in\widehat{\mathfrak T}_{G_m,s}$ be given by equation \eqref{Fhatnu}.
Then for each $\lambda \in \mathbb C_+$, the analytic paths  space integral 
$\int_{\mathcal C_0^{\mathrm{gBm}}}^{\text{\rm an}_{\lambda}}F(\mathfrak x) d \mu_{\mathcal C_0^{\mathrm{gBm}}}(\mathfrak x)$ 
exists and is  given by the formula
\begin{equation}\label{eq:Elambda}
\int_{\mathcal C_0^{\mathrm{gBm}}}^{\text{\rm an}_{\lambda}}F(\mathfrak x) d \mu_{\mathcal C_0^{\mathrm{gBm}}}(\mathfrak x)
=\int_{\mathbb R^m} \exp\bigg\{ -\frac{b(s)}{2\lambda} \sum\limits_{j=1}^m v_j^2 
    +i\lambda^{-1/2}\sum\limits_{j=1}^m (g_j,a)_{C_{a,b}'} v_j   \bigg\}d \nu(\vec v). 
\end{equation}
\end{theorem}
\begin{proof}
By \eqref{Fhatnu}, \eqref{eq:hatnu}, the Fubini theorem, \eqref{eq:well} with $n=1$,
\eqref{eq:int-formula-cab}, and  the fact that the set $\{g_1,\ldots,g_m\}$ is orthonormal 
in $C_{a,b}'[0,T]$,  we have that for all $\lambda>0$,
\[
\begin{aligned}
&J(\lambda) 
 =\int_{\mathcal C_0^{\mathrm{gBm}}} F(\lambda^{-1/2}\mathfrak x) d \mu_{\mathcal C_0^{\mathrm{gBm}}}(\mathfrak x) \\
&= \int_{\mathbb R^m} \int_{\mathcal C_0^{\mathrm{gBm}}} 
     \exp\bigg\{i\lambda^{-1/2}\sum\limits_{j=1}^m
     (g_j, \mathfrak x(s))^{\sim} v_j \bigg\}d\mu_{\mathcal C_0^{\mathrm{gBm}}}(\mathfrak x)d\nu(\vec v) \\
&= \int_{\mathbb R^m} \int_{C_{a,b}[0,T]}
     \exp\bigg\{i\lambda^{-1/2}\sum\limits_{j=1}^m
      \big(g_j, \sqrt{b(s)}(x-a)+a\big)^{\sim}v_j \bigg\}d\mu(x) d \nu(\vec v)\\
&= \int_{\mathbb R^m} \int_{C_{a,b}[0,T]}
     \exp\bigg\{i\lambda^{-1/2}\bigg(\sqrt{b(s)}\sum\limits_{j=1}^mg_jv_j,x\bigg)^{\sim} \bigg\}d\mu(x)\\
&\qquad\quad\times
     \exp\bigg\{-i\lambda^{-1/2}\sqrt{b(s)}\sum\limits_{j=1}^m(g_j,a)_{C_{a,b}'}v_j
                +i\lambda^{-1/2}\sum\limits_{j=1}^m(g_j,a)_{C_{a,b}'}v_j\bigg\}  \nu(\vec v)\\
\end{aligned}
\]
\[
\begin{aligned}
&=\int_{\mathbb R^m}  
  \exp\bigg\{-\frac{1}{2\lambda}\bigg\|\sqrt{b(s)}\sum\limits_{j=1}^m  g_jv_j \bigg\|_{C_{a,b}'}^2 
 +i\lambda^{-1/2}\bigg(\sqrt{b(s)}\sum\limits_{j=1}^mg_jv_j,a\bigg)_{C_{a,b}'} \bigg\}  \\
&\qquad\quad\times
     \exp\bigg\{-i\lambda^{-1/2}\sqrt{b(s)}\sum\limits_{j=1}^m(g_j,a)_{C_{a,b}'}v_j
                +i\lambda^{-1/2}\sum\limits_{j=1}^m(g_j,a)_{C_{a,b}'}v_j\bigg\}  \nu(\vec v)\\
&=\int_{\mathbb R^m}  
  \exp\bigg\{-\frac{b(s)}{2\lambda}\sum\limits_{j=1}^m\sum\limits_{k=1}^m (g_j,g_k)_{C_{a,b}'}v_jv_k 
                +i\lambda^{-1/2}\sum\limits_{j=1}^m(g_j,a)_{C_{a,b}'}v_j\bigg\}  \nu(\vec v)\\
&=  \int_{\mathbb R^m} \exp\bigg\{-\frac{b(s)}{2\lambda}\sum\limits_{j=1}^m v_j^2 
    + i \lambda^{-1/2}\sum\limits_{j=1}^m(g_j, a)_{C_{a,b}'} v_j    \bigg\} 
    d \nu( \vec v ). 
\end{aligned}
\]

\par
Now let 
\[
J^*(\lambda) =\int_{\mathbb R^m} \exp\bigg\{ -\frac{b(s)}{2\lambda}\sum_{j=1}^m v_j^2 
    + i \lambda^{-1/2}\sum_{j=1}^m(g_j, a)_{C_{a,b}'} v_j    \bigg\} d \nu( \vec v )
\] 
for $\lambda \in \mathbb C_+$. Then $J^*(\lambda)=J(\lambda)$ for all $\lambda>0$.

We will use the Morera theorem to show that $J^*(\lambda)$ is  analytic on $\mathbb C_+$.
First, let $\{\lambda_l\}_{l=1}^{\infty}$ be a sequence   in $\mathbb C_+$ such that  
$\lambda_l \to \lambda$ through $\mathbb C_+$. Then $\lambda_l^{-1/2}\to \lambda^{-1/2}$
and $\text{Re}(\lambda_l )\ne 0$ for all $l\in \mathbb N$. Thus we have that for 
each $\l\in \mathbb N$,
\[
\begin{aligned}
&\bigg| \exp\bigg\{-\frac{b(s)}{2\lambda_l}
    \sum\limits_{j=1}^m v_j^2 
    + i \lambda_l^{-1/2}\sum\limits_{j=1}^m(g_j, a)_{C_{a,b}'} v_j \bigg\}  \bigg|\\
& = \exp\bigg\{ -\frac{b(s)\textrm{Re}(\lambda_l)}{2|\lambda_l |^2}\sum\limits_{j=1}^m v_j^2
    - \textrm{Im}(\lambda_{l}^{-1/2}) \sum\limits_{j=1}^m(g_j, a)_{C_{a,b}'} v_j  \bigg\}\\
& = \exp\bigg\{ -\frac12 \sum\limits_{j=1}^m 
   \bigg( \frac{\sqrt{b(s)\text{Re}(\lambda_l)}}{ |\lambda_l | }v_j  
   +\frac{|\lambda_l|  \text{Im}(\lambda_l^{-1/2})  }{\sqrt{b(s)\text{Re}(\lambda_l)}}
   (g_j, a)_{C_{a,b}'} \bigg)^2  \\
& \qquad\qquad\qquad\qquad \qquad  
 +\frac12 \sum\limits_{j=1}^m
    \frac{|\lambda_l|^2(\text{Im}(\lambda_l^{-1/2}))^2}{b(s)\text{Re}(\lambda_l)}
    (g_j, a)_{C_{a,b}'}^2\bigg\}\\
&\le  \exp\bigg\{\frac{|\lambda_l|^2(\text{Im}(\lambda_l^{-1/2}))^2}{2b(s)\text{Re}(\lambda_l)}
  \sum\limits_{j=1}^m(g_j, a)_{C_{a,b}'}^2\bigg\}.
\end{aligned}
\] 
Since $\nu\in \mathcal{ M} (\mathbb R^m)$,   we see that
\[
\begin{aligned}
&\bigg| \int_{\mathbb R^m} 
  \exp\bigg\{    \frac{|\lambda_l|^2(\text{Im}(\lambda_l^{-1/2}))^2}{2b(s)\text{Re}(\lambda_l)}
  \sum\limits_{j=1}^m(g_j, a)_{C_{a,b}'}^2\bigg\} d \nu(\vec v)\bigg| \\
&\le \exp\bigg\{    \frac{|\lambda_l|^2(\text{Im}(\lambda_l^{-1/2}))^2}{2b(s)\text{Re}(\lambda_l)}
  \sum\limits_{j=1}^m(g_j, a)_{C_{a,b}'}^2\bigg\} \|\nu\|< + \infty 
\end{aligned}
\]
for each  $l \in \mathbb N$.  Furthermore we have that
\[
\begin{aligned}
&\lim\limits_{l\to\infty}\int_{\mathbb R^m}
 \exp\bigg\{    \frac{|\lambda_l|^2(\text{Im}(\lambda_l^{-1/2}))^2}{2b(s)\text{Re}(\lambda_l)} 
  \sum\limits_{j=1}^m (g_j, a)_{C_{a,b}'}^2 \bigg\} d \nu(\vec v)  \\
& =\lim\limits_{l\to\infty}
\exp\bigg\{    \frac{|\lambda_l|^2(\text{Im}(\lambda_l^{-1/2}))^2}{2b(s)\text{Re}(\lambda_l)}
  \sum\limits_{j=1}^m(g_j, a)_{C_{a,b}'}^2\bigg\} \nu(\mathbb R^m)  \\
&=\exp\bigg\{    \frac{|\lambda |^2(\text{Im}(\lambda^{-1/2}))^2}{2b(s)\text{Re}(\lambda )}
  \sum\limits_{j=1}^m(g_j, a)_{C_{a,b}'}^2\bigg\} \nu(\mathbb R^m)\\
& =  \int_{\mathbb R^m}
\exp\bigg\{    \frac{|\lambda |^2(\text{Im}(\lambda^{-1/2}))^2}{2b(s)\text{Re}(\lambda )}
\sum\limits_{j=1}^m(g_j, a)_{C_{a,b}'}^2\bigg\} d \nu(\vec v).  
\end{aligned}
\]
Thus, by Theorem 4.17 in \cite[p.92]{Royden88}, $J^*(\lambda)$  is continuous on $\mathbb C_+$.
Next, using the fact that
\[
k(\lambda)\equiv\exp\bigg\{ -\frac{b(s)}{2 \lambda}\sum_{j=1}^mv_j^2 
+i \lambda^{-1/2}\sum_{j=1}^m (g_j,a)_{C_{a,b}'}v_j\bigg\} 
\]
is analytic on $\mathbb C_+$, and  applying the Fubini theorem, it follows that
\[
\int_{\triangle} J^*(\lambda) d \lambda
=\int_{\mathbb R^m} \int_{\triangle} k(\lambda) d \lambda d\nu(\vec v)=0
\]
for all rectifiable simple closed curve $\triangle$ lying $\mathbb C_+$. 
Thus by the Morera theorem, $J^*(\lambda)$ is   analytic on $\mathbb C_+$.
Therefore the analytic paths space  integral 
$J^*(\lambda)
=\int_{\mathcal C_0^{\mathrm{gBm}}}^{\mathrm{an}_\lambda} F(\mathfrak x) d\mu_{\mathcal C_0^{\mathrm{gBm}}}(\mathfrak  x) $
exists and   is given by   equation \eqref{eq:Elambda}.
\end{proof}

\par
The observation below will be very useful in the development
of our results for the  analytic paths space Feynman integral
of functionals given by equation \eqref{Fhatnu}.

\par
If $a(t)\equiv 0$ on $[0,T]$, 
then for all $F \in\widehat{\mathfrak T}_{G_m,s}$ given by equation \eqref{Fhatnu},
the  analytic  paths space  Feynman integral 
$\int_{\mathcal C_0^{\mathrm{gBm}}}^{\text{\rm anf}_{q}}F(\mathfrak x) d \mu_{\mathcal C_0^{\mathrm{gBm}}}(\mathfrak x)$
will always exist for all real $q \ne 0$ and be given by the formula
\[
\int_{\mathcal C_0^{\mathrm{gBm}}}^{\text{\rm anf}_{q}}F(\mathfrak x) d \mu_{\mathcal C_0^{\mathrm{gBm}}}(\mathfrak x)
=\int_{\mathbb R^m} \exp\bigg\{ -\frac{ib(s)}{2q} \sum\limits_{j=1}^m v_j^2 
\bigg\}d \nu(\vec v).
\]
However for $a(t)$ as in Section   \ref{sec:path}, and proceeding formally 
using \eqref{eq:Elambda} and \eqref{eq:def-Fint},  we observe that 
$\int_{\mathcal C_0^{\mathrm{gBm}}}^{\text{\rm anf}_{q}}F(\mathfrak x) d \mu_{\mathcal C_0^{\mathrm{gBm}}}(\mathfrak x)$ 
will be given by equation  \eqref{eq:E-feynman}  below if it exists. But the integral on the 
right-hand side of  \eqref{eq:E-feynman}  
might not exist if the real part of
\[
\bigg\{ -\frac{ib(s)}{2q} \sum\limits_{j=1}^m v_j^2 
    +i(-iq)^{-1/2}\sum\limits_{j=1}^m (g_j,a)_{C_{a,b}'} v_j   \bigg\}
\]
is positive.

\par
We emphasize that any functional $F\in \widehat{\mathfrak T}_{G_m,s}$
is bounded on $\mathcal C_0^{\mathrm{gBm}}$, because
\[
\begin{aligned}
|F(\mathfrak x)|
&=\bigg| \int_{\mathbb R^m} 
     \exp\bigg\{i\sum\limits_{j=1}^m
     (g_j, \mathfrak x(s))^{\sim} v_j \bigg\}d\nu(\vec v)\bigg|\\
&\le \int_{\mathbb R^m}
\bigg|\exp\bigg\{i\sum\limits_{j=1}^m
     (g_j, \mathfrak x(s))^{\sim} v_j \bigg\}\bigg|d|\nu|(\vec v) 
\le \int_{\mathbb R^m} d|\nu|(\vec v) 
=\|\nu\|<+\infty.
\end{aligned}
\]
However, there is a functional $F$ in  $\widehat{\mathfrak T}_{G_m,s}$ which is not 
analytic  paths space  Feynman integrable on $\mathcal C_0^{\mathrm{gBm}}$.  In order to present an 
example of the functional $F$ which is not   analytic  paths space  Feynman integrable, 
we consider the class 
\[
\widehat{\mathfrak T}_{G_1,s}
=\bigg\{F: F(x)=\widehat \nu(x)=\int_{\mathbb R}\exp\big\{i(g,\mathfrak x(s))^{\sim}v\big\}d\nu(v)
\mbox{ for some } \nu\in \mathcal M (\mathbb R)\bigg\}.
\] 
where $G_1=\{g\}$ is an orthonormal set in $C_{a,b}'[0,T]$. Next   we consider a measure 
$\alpha$ on $\mathbb R$ which is concentrated on the set of natural numbers $\mathbb N$
and is given by $\alpha(\{m\})=1/m^2$ for each $m \in \mathbb N$. Then $\alpha$ is an element 
of $\mathcal M(\mathbb R)$. Consider the functional $F\in \widehat{\mathfrak T}_{G_1,s}$
given by
\[
F(x)=\int_{\mathbb R}\exp\big\{i(g,\mathfrak x(s))^{\sim}v\big\}d\alpha(v).
\] 
In this case, by equation  \eqref{eq:E-feynman}   below, we have that for a positive 
real number $q>0$,
\begin{equation}\label{eq:series}
\begin{aligned}
&\int_{\mathcal C_0^{\mathrm{gBm}}}^{\text{\rm anf}_{q}} F(\mathfrak x)d\mu_{\mathcal C_0^{\mathrm{gBm}}}(\mathfrak x)   \\
&= \int_{\mathbb R}\exp\bigg\{-\frac{ib(s)}{2q}
 v^2+i(-iq)^{-1/2}(g ,a)_{C_{a,b}'}v\bigg\}d \alpha (v) \\
&=\sum_{m=1}^{\infty}\exp\bigg\{ -\frac{ib(s)m^2}{2q} 
 + \bigg(-\frac{1}{\sqrt{2q}}+\frac{i}{\sqrt{2q}}\bigg) (g ,a)_{C_{a,b}'}m \bigg\} \frac{1}{m^2}.
\end{aligned}
\end{equation}
Then,  we have
\[
\begin{aligned}
L
&\equiv\lim_{m\to\infty}\frac{\big|\exp\big\{ -\frac{ib(s) (m+1)^2}{2q} 
 +(-\frac{1}{\sqrt{2q}}+\frac{i}{\sqrt{2q}})(g ,a)_{C_{a,b}'}(m+1)\big\}\frac{1}{(m+1)^2}\big|}
{\big|\exp\big\{ -\frac{ib(s)m^2}{2q} 
 +(-\frac{1}{\sqrt{2q}}+\frac{i}{\sqrt{2q}})(g ,a)_{C_{a,b}'}m\big\}\frac{1}{m^2} \big|}\\
&=\lim_{m\to\infty}\frac{\exp\big\{-\frac{1}{\sqrt{2q}}(g ,a)_{C_{a,b}'}(m+1)\big\}}
{\exp\big\{-\frac{1}{\sqrt{2q}}(g ,a)_{C_{a,b}'}m\big\} } 
=\exp\bigg\{-\frac{1}{\sqrt{2q}}(g ,a)_{C_{a,b}'}\bigg\}. 
\end{aligned}
\]
If  $(g,a)_{C_{a,b}'}<0$, then  $L>1$ and so, by the d'Alembert ratio test, we see that 
the series in the last expression  of \eqref{eq:series} diverges.

\par
Consequently, in view of this example, we clearly need to impose additional restrictions
on the functionals  in $\widehat{\mathfrak T}_{G_m,s}$ to establish the existence of the 
analytic  Feynman integral on $\mathcal C_0^{\mathrm{gBm}}$.

\par 
For a positive real number $q_0$, we define a subclass $\widehat{\mathfrak T}_{G_m,s}^{q_0}$ of 
$\widehat{\mathfrak T}_{G_m,s}$ by $F\in\widehat{\mathfrak T}_{G_m,s}^{q_0}$ if and only if
\begin{eqnarray}\label{eq;condition}
\int_{\mathbb R^m} \exp\bigg\{  \frac{ \|a\|_{C_{a,b}'} }{\sqrt{|2q_0|}}
\sum\limits_{j=1}^m | v_j|   \bigg\} d |\nu|(\vec v)< + \infty.
\end{eqnarray}
where $\nu$ and $F$ are related by \eqref{Fhatnu}.

\par
Note that in case $a(t)\equiv 0$ and $b(t)=t$ on $[0,T]$, the function space $C_{a,b}[0,T]$ 
reduces to the classical Wiener space $C_0[0,T]$ and $(g_j,a)_{C_{a,b}'}=0$ for all $j=1,\ldots,n$. 
In this case, it follows that for all $q_0>0$, 
$\widehat{\mathfrak T}_{G_m,s}^{q_0}=\widehat{\mathfrak T}_{G_m,s}$.

\begin{theorem}\label{thm:feynman}
Given a positive real number  $q_0$, let  $F \in \widehat{\mathfrak T}_{G_m,s}^{q_0}$ be 
given by \eqref{Fhatnu}. Then, for all real $q$ with $|q|> |q_0|$, the  analytic  paths 
space  Feynman integral 
$\int_{\mathcal C_0^{\mathrm{gBm}}}^{\text{\rm anf}_{q}}F(\mathfrak x) d \mu_{\mathcal C_0^{\mathrm{gBm}}}(\mathfrak x)$
exists  and is given by the formula
\begin{equation}\label{eq:E-feynman}
\begin{aligned}
&\int_{\mathcal C_0^{\mathrm{gBm}}}^{\text{\rm anf}_{q}}F(\mathfrak x) d \mu_{\mathcal C_0^{\mathrm{gBm}}}(\mathfrak x)\\
&=\int_{\mathbb R^m} \exp\bigg\{ -\frac{ib(s)}{2q} \sum\limits_{j=1}^m v_j^2 
    +i(-iq)^{-1/2}\sum\limits_{j=1}^m (g_j,a)_{C_{a,b}'} v_j   \bigg\}d \nu(\vec v).
\end{aligned}
\end{equation}
\end{theorem}
\begin{proof} 
Let $\{\lambda_l\}_{l=1}^{\infty}$ be a sequence of complex numbers such that 
$\lambda_l \to-iq$ through $\mathbb C_+$ and  for each $l\in \mathbb N$, let 
\[
f_l (\vec v)=\exp\bigg\{-\frac{b(s)}{2\lambda_l}\sum_{j=1}^m v_j^2
+i\lambda_l^{-1/2}\sum_{j=1}^{n}(g_j,a)_{C_{a,b}'}v_j\bigg\}.
\] 
Then $f_l (\vec v)$ converges to 
\[
f(\vec v)\equiv  
\exp\bigg\{-\frac{ib(s)}{2q} \sum_{j=1}^m v_j^2
     +i(-iq)^{-1/2}\sum_{j=1}^{n}(g_j,a)_{C_{a,b}'}v_j\bigg\}.
\] 
By Theorem \ref{thm:analytic}, for all $l\in \mathbb N$, 
$\int_{\mathbb R^m }f_l (\vec v) d \nu(\vec v)$ exists.
Since $|\text{\rm arg}(\lambda_l^{-1/2})|< \pi/4$ for every $l\in \mathbb N$ and 
$\lambda_l^{-1/2}=\text{\rm Re}(\lambda_l^{-1/2})+i\text{\rm Im}(\lambda_l^{-1/2}) 
\to (-iq)^{-1/2}=1/\sqrt{|2q|} +i\text{\rm sign}(q)/\sqrt{|2q|}$,
we see that $\text{\rm Re}(\lambda_l^{-1/2})>|\text{\rm Im}(\lambda_l^{-1/2})|$ 
for every $l\in \mathbb N$ and there exists a sufficiently large $k\in \mathbb N$ such 
that $|\text{\rm Im}(\lambda_l^{-1/2})|<1/\sqrt{|2q|}<1/\sqrt{|2q_0|}$ for every $l\ge k$.
Thus, using the Cauchy--Schwartz inequality, it follows that for each $l\ge k$,
\[
\begin{aligned}
|f_l (\vec v)|
&=\bigg|\exp\bigg\{-\frac{b(s)}{2}
\Big([\text{\rm Re}(\lambda_l^{-1/2})]^2 -[\text{\rm Im}(\lambda_l^{-1/2})]^2 \\
& \qquad\qquad\qquad\qquad\qquad\qquad
+i\text{\rm Re}(\lambda_l^{-1/2})\text{\rm Im}(\lambda_l^{-1/2})\Big)\sum_{j=1}^mv_j^2\\
& \qquad\qquad
+i \Big(\text{\rm Re}(\lambda_l^{-1/2})+i\text{\rm Im}(\lambda_l^{-1/2})\Big)
\sum_{j=1}^m(g_j,a)_{C_{a,b}'}v_j\bigg\}\bigg|\\
&\le\exp\bigg\{-\text{\rm Im}(\lambda_l^{-1/2} )\sum\limits_{j=1}^m (g_j,a)_{C_{a,b}'}v_j \bigg\}\\
&\le  \exp \bigg\{|\text{\rm Im}(\lambda_l^{-1/2} )| \|a\|_{C_{a,b}'} \sum\limits_{j=1}^m | v_j|  \bigg\} \\
\end{aligned}
\]
\[
\begin{aligned}
&  <   \exp \bigg\{\frac{ \|a\|_{C_{a,b}'} }{\sqrt{|2q_0|}}
 \sum\limits_{j=1}^m | v_j|  \bigg\}
\end{aligned}
\]
and so, by condition \eqref{eq;condition},
\begin{equation*}
\begin{aligned}
\bigg|\int_{\mathbb R^m} f_l(\vec v) d \nu (\vec v) \bigg| 
&\le\int_{\mathbb R^m} \big|f_l(\vec v)\big| |d \nu (\vec v)|  \\
&  <  \int_{\mathbb R^m}  \exp \bigg\{\frac{ \|a\|_{C_{a,b}'} }{\sqrt{|2q_0|}}
\sum\limits_{j=1}^m | v_j|  \bigg\}  d |\nu|(\vec v)   < +\infty.
\end{aligned}\end{equation*}
Also, by condition \eqref{eq;condition}, we have
\begin{equation*}
\begin{aligned}
\bigg|\int_{\mathbb R^m} f(\vec v) d \nu (\vec v) \bigg| 
&\le \int_{\mathbb R^m} \exp \bigg\{\frac{ \|a\|_{C_{a,b}'} }{\sqrt{|2q|}}
\sum\limits_{j=1}^m | v_j|  \bigg\} d |\nu|(\vec v)  \\
& <  \int_{\mathbb R^m} \exp \bigg\{\frac{ \|a\|_{C_{a,b}'} }{\sqrt{|2q_0|}}
\sum\limits_{j=1}^m | v_j|  \bigg\} d |\nu|(\vec v)   < +\infty.
\end{aligned}\end{equation*}
Thus by the dominated convergence theorem, 
it follows  equation \eqref{eq:E-feynman}.
\end{proof}

 \subsection{Functionals in $\widehat{\mathfrak T}_{G_m,\vec s}$}\label{sec:cylinder}

\par
Let $n$ and $m$ be positive integers. Given an $n$-tuple   $\vec s=(s_1,\ldots, s_n)$ with 
$0=s_0< s_1<\cdots <s_n \le T$ and an  orthonormal set  $G_m=\{g_1,\ldots, g_m\}$  of functions 
in $C_{a,b}'[0,T]$, let $\widehat{\mathfrak T}_{G_m,\vec s}$ be the space of all functionals 
$F$ on $\mathcal C_0^{\mathrm{gBm}}$ of the form
\begin{equation}\label{eq:Fhatmn}
F(\mathfrak x)=\widehat \nu(( g_1 ,\mathfrak x(s_1) )^{\sim},\ldots, ( g_m,\mathfrak x(s_1))^{\sim},
\ldots, ( g_1 ,\mathfrak x(s_n) )^{\sim},\ldots, ( g_m,\mathfrak x(s_n))^{\sim})
\end{equation}
for  $\mathfrak x\in \mathcal C_0^{\mathrm{gBm}}$, where $\nu$ is an element of $\mathcal M(\mathbb R^{mn})$,
the class of all complex-valued Borel measures on $\mathbb R^{mn}$ with finite total variation,
and $\widehat \nu$ denotes the Fourier--Stieltjes transform of $\nu$ given by
\begin{equation}\label{eq:hatnumn}
\widehat\nu (\vec u)
=\int_{\mathbb R^{mn}} \exp\bigg\{ i \sum\limits_{j=1}^m\sum\limits_{k=1}^n u_{j,k} v_{j,k} \bigg\}
d \nu (\vec v)
\end{equation}
where  
\[
\vec u=(u_{1,1},\ldots,u_{1,n},u_{2,1},\ldots,u_{2,n},\ldots,u_{m,1},\ldots,u_{m,n})\in\mathbb R^{mn}
\]
and
\[
\vec v=(v_{1,1},\ldots,v_{1,n},v_{2,1},\ldots,v_{2,n},\ldots,v_{m,1},\ldots,v_{m,n})\in\mathbb R^{mn}.
\]
 Also, for a positive real number $q_0$, we define a subclass $\widehat{\mathfrak T}_{G_m,\vec s}^{q_0}$ 
of $\widehat{\mathfrak T}_{G_m,\vec s}$ by $F\in\widehat{\mathfrak T}_{G_m,\vec s}^{q_0}$ if and only if
\[
\int_{\mathbb R^{mn}} \exp\bigg\{  \frac{ \|a\|_{C_{a,b}'} }{\sqrt{|2q_0|}}
\sum\limits_{j=1}^m\sum\limits_{k=1}^n | v_{j,k}|   \bigg\} d |\nu|(\vec v)< + \infty.
\]
where $\nu$ and $F$ are related by \eqref{eq:Fhatmn}.

\par
Our next theorem shows the analytic paths  space integral exists for all  
$F\in\widehat{\mathfrak T}_{G_m,\vec s}$. The following summation formula
\begin{equation}\label{eq:sumsum}
\sum_{k=1}^n\sum_{l=1}^k A_lB_k=\sum_{l=1}^n\sum_{k=l}^n A_lB_k
\end{equation}
will be helpful to prove the theorem.

\begin{theorem}\label{thm:analytic-tuple} 
Let $F\in\widehat{\mathfrak T}_{G_m,\vec s}$ be given by equation \eqref{eq:Fhatmn}.
Then for each $\lambda \in \mathbb C_+$, the analytic paths  space integral 
$\int_{\mathcal C_0^{\mathrm{gBm}}}^{\text{\rm an}_{\lambda}}F(\mathfrak x) d \mu_{\mathcal C_0^{\mathrm{gBm}}}(\mathfrak x)$ 
exists  and is  given by the formula
\[
\begin{aligned}
\int_{\mathcal C_0^{\mathrm{gBm}}}^{\text{\rm an}_{\lambda}}
F(\mathfrak x) d \mu_{\mathcal C_0^{\mathrm{gBm}}}(\mathfrak x)
&=\int_{\mathbb R^{mn}} 
\exp\bigg\{-\frac{1}{2\lambda}\sum\limits_{j=1}^m\sum_{l=1}^n
 [b(s_l)-b(s_{l-1})]\bigg( \sum\limits_{k=l}^n v_{j,k}\bigg)^2\\
&\qquad\qquad\qquad
      +i\lambda^{-1/2}\sum\limits_{j=1}^m\sum\limits_{k=1}^n
     (g_j,a)_{C_{a,b}'}v_{j,k}\bigg\}  \nu(\vec v).
\end{aligned}
\]
\end{theorem}
\begin{proof}
By \eqref{eq:Fhatmn}, \eqref{eq:hatnumn}, the Fubini theorem, \eqref{eq:well} together with 
\eqref{eq:Tt}, and  \eqref{eq:Ls}, we first obtain that for all $\lambda>0$,
\begin{equation}\label{eq:last-lambda}
\begin{aligned}
 J(\lambda) 
& =\int_{\mathcal C_0^{\mathrm{gBm}}} F(\lambda^{-1/2}\mathfrak x)d\mu_{\mathcal C_0^{\mathrm{gBm}}}(\mathfrak x)  \\
&= \int_{\mathbb R^{mn}} \int_{\mathcal C_0^{\mathrm{gBm}}} 
     \exp\bigg\{i\lambda^{-1/2}\sum\limits_{j=1}^m\sum\limits_{k=1}^n
     (g_j, \mathfrak x(s_k))^{\sim} v_{j,k}  \bigg\}d\mu_{\mathcal C_0^{\mathrm{gBm}}}(\mathfrak x)d\nu(\vec v) \\
&= \int_{\mathbb R^{mn}} \int_{C_{a,b}^n[0,T]}
     \exp\bigg\{i\lambda^{-1/2}\sum\limits_{j=1}^m\sum\limits_{k=1}^n
      \big(g_j, L_{\vec{s},k}(\vec{x})\big)^{\sim}v_{j,k}  \bigg\}d\mu^n(\vec x) d \nu(\vec v)\\
\end{aligned}
\end{equation}
\[
\begin{aligned}
&= \int_{\mathbb R^{mn}} \int_{C_{a,b}^n[0,T]}\\
&\qquad\quad\times
     \exp\bigg\{i\lambda^{-1/2} \sum\limits_{j=1}^m\sum\limits_{k=1}^n
     \bigg[\sum\limits_{l=1}^k\sqrt{b(s_l)-b(s_{l-1})}(g_j,x_l)^{\sim}\bigg]v_{j,k}\bigg\}d\mu^n(\vec x)\\
&\quad\times
     \exp\bigg\{-i\lambda^{-1/2}\sum\limits_{j=1}^m\sum\limits_{k=1}^n
      \bigg[\sum\limits_{l=1}^k\sqrt{b(s_l)-b(s_{l-1})}(g_j,a)_{C_{a,b}'}\bigg]v_{j,k}\\
&\qquad\qquad\,\,\,
      +i\lambda^{-1/2}\sum\limits_{j=1}^m\sum\limits_{k=1}^n(g_j,a)_{C_{a,b}'}v_{j,k}\bigg\} d \nu(\vec v).
\end{aligned}
\]
But, using \eqref{eq:sumsum}, the first  and the second triple summations
in the last expression of \eqref{eq:last-lambda} can be rewritten by
\begin{equation}\label{eq:ssvar}
\begin{aligned}
& \sum\limits_{j=1}^m\sum\limits_{k=1}^n\bigg[\sum\limits_{l=1}^k
   \sqrt{b(s_l)-b(s_{l-1})}(g_j,x_l)^{\sim}\bigg]v_{j,k} \\
&=\sum\limits_{j=1}^m\sum\limits_{l=1}^n\bigg[\sum\limits_{k=l}^n
   \sqrt{b(s_l)-b(s_{l-1})}(g_j,x_l)^{\sim}\bigg]v_{j,k}\\ 
&=\sum\limits_{l=1}^n\sum\limits_{j=1}^m\bigg[\sum\limits_{k=l}^n
   \sqrt{b(s_l)-b(s_{l-1})}(g_j,x_l)^{\sim}\bigg]v_{j,k} 
\end{aligned}
\end{equation}
and 
\begin{equation}\label{eq:ssmean}
\begin{aligned}
& \sum\limits_{j=1}^m\sum\limits_{k=1}^n\bigg[\sum\limits_{l=1}^k
   \sqrt{b(s_l)-b(s_{l-1})}(g_j,a)_{C_{a,b}'}\bigg]v_{j,k}\\
&=\sum\limits_{j=1}^m\sum\limits_{l=1}^n\bigg[\sum\limits_{k=l}^n
   \sqrt{b(s_l)-b(s_{l-1})}(g_j,a)_{C_{a,b}'}\bigg]v_{j,k}\\
&=\sum\limits_{l=1}^n\sum\limits_{j=1}^m\bigg[\sum\limits_{k=l}^n
   \sqrt{b(s_l)-b(s_{l-1})}(g_j,a)_{C_{a,b}'}\bigg]v_{j,k},
\end{aligned}
\end{equation}
respectively. Using \eqref{eq:ssvar}, \eqref{eq:ssmean}, \eqref{eq:int-formula-cab}, 
and the fact that the set $\{g_1,\ldots,g_n\}$ is orthonormal in $C_{a,b}'[0,T]$,
it follows that for all $\lambda>0$,
\begin{equation}\label{eq:last-lambda-II}
\begin{aligned}
J(\lambda)  
&= \int_{\mathbb R^{mn}}\Bigg( \int_{C_{a,b}^n[0,T]}\\
&\quad\quad \times
     \exp\bigg\{i\lambda^{-1/2} 
    \sum\limits_{l=1}^n\sum\limits_{j=1}^m\bigg[\sum\limits_{k=l}^n
   \sqrt{b(s_l)-b(s_{l-1})}(g_j,x_l)^{\sim}\bigg]v_{j,k} \bigg\}d\mu^n(\vec x)\Bigg)\\
\end{aligned}
\end{equation}
\[
\begin{aligned}
&\quad\quad\quad\times
     \exp\bigg\{-i\lambda^{-1/2}
    \sum\limits_{l=1}^n\sum\limits_{j=1}^m\bigg[\sum\limits_{k=l}^n
   \sqrt{b(s_l)-b(s_{l-1})}(g_j,a)_{C_{a,b}'}\bigg]v_{j,k}\\
&\quad\qquad\qquad\,\,
      +i\lambda^{-1/2}\sum\limits_{j=1}^m\sum\limits_{k=1}^n
    (g_j,a)_{C_{a,b}'}v_{j,k}\bigg\} d \nu(\vec v)\\
&= \int_{\mathbb R^{mn}}\Bigg(\prod_{l=1}^n\bigg[\int_{C_{a,b}[0,T]}\\
&\quad\qquad\times
\exp\bigg\{i\lambda^{-1/2} \bigg(\sum\limits_{j=1}^m\sum\limits_{k=l}^n
\sqrt{b(s_l)-b(s_{l-1})}v_{j,k} g_j,x_l\bigg)^{\sim}\bigg\}d\mu (x_l)\bigg]\Bigg)\\
&\quad\quad\quad\times
    \exp\bigg\{-i\lambda^{-1/2}\sum\limits_{l=1}^n\sum\limits_{j=1}^m
    \bigg[\sum\limits_{k=l}^n\sqrt{b(s_l)-b(s_{l-1})}(g_j,a)_{C_{a,b}'}\bigg]v_{j,k}\\
&\quad\qquad\qquad\,\,
      +i\lambda^{-1/2}\sum\limits_{j=1}^m\sum\limits_{k=1}^n(g_j,a)_{C_{a,b}'}v_{j,k}\bigg\}  
d\nu(\vec v)\\
&=\int_{\mathbb R^{mn}}\Bigg(\prod_{l=1}^n\bigg[
\exp\bigg\{-\frac{1}{2\lambda}\bigg\|\sum\limits_{j=1}^m\sum\limits_{k=l}^n
  \sqrt{b(s_l)-b(s_{l-1})}v_{j,k}g_j\bigg\|_{C_{a,b}'}^2\\
&\qquad \qquad\qquad\qquad  
+i\lambda^{-1/2} \bigg(\sum\limits_{j=1}^m\sum\limits_{l=1}^k
    \sqrt{b(s_l)-b(s_{l-1})}v_{j,k} g_j,a\bigg)_{C_{a,b}'}\bigg\}\bigg]\Bigg) \\
&\qquad\,\,\,\times
    \exp\bigg\{-i\lambda^{-1/2}\sum\limits_{l=1}^n\sum\limits_{j=1}^m
     \bigg[\sum\limits_{k=l}^n\sqrt{b(s_l)-b(s_{l-1})}(g_j,a)_{C_{a,b}'}\bigg]v_{j,k}\\
&\qquad\qquad\qquad\,\,
      +i\lambda^{-1/2}\sum\limits_{j=1}^m\sum\limits_{k=1}^n
     (g_j,a)_{C_{a,b}'}v_{j,k}\bigg\} d \nu(\vec v)\\
&=\int_{\mathbb R^{mn}} 
\exp\bigg\{-\frac{1}{2\lambda}\sum_{l=1}^n\sum\limits_{j=1}^m\bigg(\sum\limits_{k=l}^n
  \sqrt{b(s_l)-b(s_{l-1})}v_{j,k}\bigg)^2\\
&\qquad \qquad\qquad 
      +i\lambda^{-1/2}\sum\limits_{j=1}^m\sum\limits_{k=1}^n
     (g_j,a)_{C_{a,b}'}v_{j,k}\bigg\} d \nu(\vec v)\\
&=\int_{\mathbb R^{mn}} 
\exp\bigg\{-\frac{1}{2\lambda}\sum\limits_{j=1}^m\sum_{l=1}^n
[b(s_l)-b(s_{l-1})]\bigg( \sum\limits_{k=l}^n v_{j,k}\bigg)^2\\
&\qquad \qquad\qquad 
      +i\lambda^{-1/2}\sum\limits_{j=1}^m\sum\limits_{k=1}^n
     (g_j,a)_{C_{a,b}'}v_{j,k}\bigg\} d \nu(\vec v).
\end{aligned}
\]

\par 
Next,  let $J^*(\lambda)$ be given by the last expression of \eqref{eq:last-lambda-II}
for each $\lambda\in \mathbb C_+$. 
Then using the techniques similar to those used in the proof of Theorem \ref{thm:analytic},
we can show that  $J^*(\lambda)=J(\lambda)$ for all $\lambda>0$ 
and $J^*(\lambda)$ is   analytic on $\mathbb C_+$.
This completes the proof.
\end{proof}

\par
Our next  theorem  follows quite readily from the techniques developed in
the proof of Theorem \ref{thm:feynman}.

\begin{theorem} 
Given a positive real number  $q_0$, let  
$F \in \widehat{\mathfrak T}_{G_m,\vec s}^{q_0}$ be given by  \eqref{eq:Fhatmn}.
Then, for all real $q$ with $|q|> |q_0|$, the generalized analytic Feynman integral 
$\int_{\mathcal C_0^{\mathrm{gBm}}}^{\text{\rm anf}_{q}}F(\mathfrak x) d \mu_{\mathcal C_0^{\mathrm{gBm}}}(\mathfrak x)$
exists  and is given by the formula
\[
\begin{aligned}
&\int_{\mathcal C_0^{\mathrm{gBm}}}^{\text{\rm anf}_{q}}F(\mathfrak x) d \mu_{\mathcal C_0^{\mathrm{gBm}}}(\mathfrak x)
=\int_{\mathbb R^{mn}} \exp\bigg\{-\frac{i}{2q}\sum\limits_{j=1}^m\sum_{l=1}^n
[b(s_l)-b(s_{l-1})]\bigg( \sum\limits_{k=l}^n v_{j,k}\bigg)^2\\
&\qquad\qquad\qquad\qquad\qquad\qquad\qquad
     +i(-iq)^{-1/2}\sum\limits_{j=1}^m\sum\limits_{k=1}^n
     (g_j,a)_{C_{a,b}'}v_{j,k}\bigg\}d \nu(\vec v). 
\end{aligned}
\]
\end{theorem}


\end{document}